\documentclass[11pt]{article}
\usepackage[a4paper,hmargin={2.5cm,2.5cm},vmargin={2.5cm,2.5cm}]{geometry}

\usepackage[utf8]{inputenc}
\usepackage[T1]{fontenc}
\usepackage{lmodern}
\usepackage{mathrsfs}
\usepackage{pgfplots}
\pgfplotsset{compat=1.18}
\usepackage[english]{babel}
\usepackage[hypertexnames=false]{hyperref}
\usepackage[shortlabels]{enumitem}
\usepackage{csquotes}
\usepackage{xcolor}
\usepackage{cite}
\usepackage[leqno]{amsmath}
\usepackage{amssymb,amsthm}
\usepackage{tikz-cd}
\usepackage{dsfont}
\usepackage{graphicx}

\allowdisplaybreaks

\numberwithin{equation}{section}

\theoremstyle{plain}
\newtheorem{theorem}{Theorem}[section]
\newtheorem{lemma}[theorem]{Lemma}

\newtheorem{corollary}[theorem]{Corollary}
\newtheorem{definition}[theorem]{Definition}

\newtheorem{remark}[theorem]{Remark}

\newcommand{\Jb}{\mathcal J_{\lambda,\beta}}
\newcommand{\Nb}{\mathcal N_{\lambda,\beta}}
\newcommand{\I}{\mathcal I_{\lambda_i}}
\newcommand{\field}[1]{\mathbb{#1}}
\newcommand{\N}{\field{N}}

\newcommand{\R}{\field{R}}

\newcommand{\wto}{\rightharpoonup}

\tikzset{%
    symbol/.style={%
        draw=none,
        every to/.append style={%
            edge node={node [sloped, allow upside down, auto=false]{$#1$}}}
    }
}
\tikzset{shorten <>/.style={shorten >=#1,shorten <=#1}}

\title{Quasilinear Elliptic Cooperative and Competitive Systems}
\author{
Annamaria Canino\\
\small Dipartimento di Matematica e Informatica, Universit\`a della Calabria,\\
\small Ponte Pietro Bucci, cubo 31B, 87036 Arcavacata di Rende, Cosenza, Italy\\
\small \texttt{annamaria.canino@unical.it}
\and
Simone Mauro\\
\small Dipartimento di Matematica e Informatica, Universit\`a della Calabria,\\
\small Ponte Pietro Bucci, cubo 31B, 87036 Arcavacata di Rende, Cosenza, Italy\\
\small \texttt{simone.mauro@unical.it}
}
\date{\today}

\begin{document}

\maketitle

\begin{abstract}
We study the existence and multiplicity of weak solutions for the following quasilinear elliptic system:
\[
\resizebox{\linewidth}{!}{$\displaystyle
\begin{cases}
-\mathrm{div}(A_1(x,u_1)\nabla u_1) + \displaystyle\frac{1}{2} D_{u_1}A_1(x,u_1)\nabla u_1 \cdot \nabla u_1 = \lambda_1 u_1 + g_{\beta,1}(u) & \text{in } \Omega, \\[3mm]
-\mathrm{div}(A_2(x,u_2)\nabla u_2) + \displaystyle\frac{1}{2} D_{u_2}A_2(x,u_2)\nabla u_2 \cdot \nabla u_2 = \lambda_2 u_2 + g_{\beta,2}(u) & \text{in } \Omega, \\[2mm]
u_1 = u_2 = 0 & \text{on } \partial\Omega,
\end{cases}
$}
\]
where $\lambda_1, \lambda_2 < \mu_1$, $\mu_1$ is the first Dirichlet eigenvalue of the Laplacian, and $\Omega$ is a bounded domain. 
The nonlinearity is derived from a potential $G_\beta$ with subcritical growth.

 We prove the existence of least energy solutions in both the cooperative ($\beta > 0$) and competitive ($\beta < 0$) regimes.
Due to the lack of differentiability of the associated energy functional, we employ nonsmooth critical point theory and variational methods based on the concept of weak slope. 

\end{abstract}


\section{Introduction} 
Let $\Omega\subset\R^N$ be a bounded domain and let $G_\beta:\R^2\to\R$ be the $C^1$-function defined by
\begin{equation}\tag{$g.1$}\label{g.1}
  G_\beta(t_1,t_2)=\frac{1}{p}\left(|t_1|^p+2\beta|t_1|^{\frac p2}|t_2|^{\frac p2}+|t_2|^p\right),\quad \beta\in\R,\ p\in(2,2^*).  
\end{equation}
Setting $g_\beta=(g_{\beta,1},g_{\beta,2})=\nabla G_\beta$, i.e.,
\[
g_\beta(t_1,t_2)=
\begin{bmatrix}
    |t_1|^{p-2}t_1+\beta|t_1|^{\frac p2-2}t_1|t_2|^{\frac p2}\\
    |t_2|^{p-2}t_2+\beta|t_2|^{\frac p2-2}t_2|t_1|^{\frac p2}
\end{bmatrix},
\]
we consider the quasilinear elliptic system
{\footnotesize
\begin{equation}\tag{$\mathcal P_\beta$}\label{Qb}
\begin{cases}
   -\text{div}(A_1(x,u_1)\nabla u_1)+\frac12 D_{u_1}A_1(x,u_1)\nabla u_1\cdot\nabla u_1=\lambda_1u_1+g_{\beta,1}(u)&\text{in $\Omega$},\\
     -\text{div}(A_2(x,u_2)\nabla u_2)+\frac12 D_{u_2}A_2(x,u_2)\nabla u_2\cdot\nabla u_2=\lambda_2u_2+g_{\beta,2}(u)&\text{in $\Omega$},\\
    u_1=u_2=0&\text{on $\partial\Omega$},
\end{cases}
\end{equation}
}
where $\lambda_1,\lambda_2<\mu_1$, and $\mu_1$ denotes the first eigenvalue of the Laplacian with Dirichlet boundary conditions.

Throughout the paper, we assume that $A_i(x,s)$ are symmetric matrices whose entries $(a_i)_{j,k}:\Omega\times\R\to\R$, for $j,k=1,\dots,N$ and $i=1,2$, are $C^1$-Carathéodory functions such that there exist $\nu\in(0,1]$ and $C_0>0$ satisfying
{\small
\begin{align}
\tag{$a.1$}\label{a.0}
&|(a_i)_{j,k}(x,s)|, |D_s(a_i)_{j,k}(x,s)|\le C_0,\quad j,k=1,\dots,N,\\
\tag{$a.2$}\label{a.1}
&A_i(x,s)\xi\cdot\xi\ge\nu|\xi|^2,\\
\tag{$a.3$}\label{a.2}
   &  0\le sD_{s}A_i(x,s)\xi\cdot\xi\le\gamma A_i(x,s)\xi\cdot\xi,\ \ \gamma\in(0,p-2),\\
   \tag{$a.4$}\label{a.3}
   &(0,+ \infty)\ni s\mapsto s^{3-p}D_sA_i(x,s)\xi\cdot\xi\ \ \text{is strictly decreasing for every $\xi\in\R^N$},
\end{align}
}
 for a.e. $x\in\Omega$ and every $(s,\xi)\in\R\times\R^N$ and for every $i=1,2$. 
 
We emphasize that assumption \eqref{a.3} is needed exclusively to handle the
competitive case $\beta<0$ and can be omitted in the cooperative regime
$\beta>0$.

We point out that if $A_i(x,s)$ is the identity matrix, we obtain the regular
semilinear case of the Laplacian:
\[
\begin{cases}
    -\Delta u_1=\lambda_1u_1+|u_1|^{p-2}u_1+\beta|u_1|^{\frac p2-2}u_1|u_2|^{\frac p2}&\text{in $\Omega$},\\
    -\Delta u_2=\lambda_2u_2+|u_2|^{p-2}u_2+\beta|u_2|^{\frac p2-2}u_2|u_1|^{\frac p2}&\text{in $\Omega$},\\
    u_1=u_2=0&\text{on $\partial\Omega$}.
\end{cases}
\]

This semilinear system has attracted considerable attention over the last two
decades. It arises in several physical applications; for instance, when
$p=4$, the complex-valued functions
$\psi_1(t,x)=e^{i\lambda_1t}u_1(x)$ and
$\psi_2(t,x)=e^{i\lambda_2t}u_2(x)$ represent solitary wave solutions to the
cubic Schr\"odinger system in nonlinear optics:

\[
\begin{cases}
    i\frac{\partial \psi_1}{\partial t} - \Delta \psi_1 = |\psi_1|^2 \psi_1 + \beta |\psi_2|^2 \psi_1 & \text{in } \Omega \times (0, +\infty), \\
    i\frac{\partial \psi_2}{\partial t} - \Delta \psi_2 = |\psi_2|^2 \psi_2 + \beta |\psi_1|^2 \psi_2 & \text{in } \Omega \times (0, +\infty).
\end{cases}
\]

The parameter $\beta\in\R$ represents the interaction between the two waves,
with positive values indicating attractive interaction and negative values
repulsive interaction.

For Dirichlet boundary conditions in the semilinear regime, we refer, among
many contributions, to
\cite{ambrosetti2007standing,chen2012positive,chen2015positive,clapp2019simple,clapp2022solutions,maia2006positive,mandel2015minimal,noris2010existence,sirakov2007least,soave2016new,tavares2010tesi}
and the references therein, while Neumann boundary conditions have been
addressed more recently, see, e.g., \cite{mauro2026least}.

In the cooperative regime $\beta>0$, a nontrivial solution can be
obtained by the Mountain Pass Theorem or, equivalently, by minimizing the
energy functional on the Nehari manifold
{\small
\[
\mathcal M_{\lambda,\beta}:=
\left\{\begin{aligned}
u\in H_0^1(\Omega;\R^2)\setminus\{0\}:\quad
&\sum_{i=1}^2\left(\int_\Omega|\nabla u_i|^2
-\lambda_i\int_\Omega u_i^2\right)\\
&=\int_\Omega\left(|u_1|^p+|u_2|^p
+2\beta|u_1|^{\frac p2}|u_2|^{\frac p2}\right)
\end{aligned}\right\}.
\]
}
For $\beta>0$ sufficiently large, a comparison with the least energy levels
of the scalar problems shows that the resulting solutions are fully
nontrivial. By contrast, for $\beta<0$, as well as for $\beta>0$ sufficiently
small, solutions at the mountain-pass level may be semi-trivial. To obtain a
fully nontrivial least energy solution, one is therefore led to minimize on
the componentwise Nehari set $\Nb$, where the Nehari identity is imposed
separately on each component. 

The study of general \emph{quasilinear} elliptic problems introduces significant
additional mathematical challenges. The existence of multiple solutions for
a broader class of quasilinear elliptic scalar equations of the form
\[
-\operatorname{div}\!\left(A(x,u)\nabla u\right)
\;+\;\frac12 D_u A(x,u)\,\nabla u\cdot\nabla u
= g(x,u),
\]
where $g(x,u)$ is a $C^1$-Carathéodory function satisfying an
Ambrosetti--Rabinowitz type condition, has been established in
\cite{caninoquasilineare1,caninoserdica,canino2025neumann}.
Quasilinear elliptic systems have also been investigated in
\cite{arioli2000existence,candela2021existence,candela2022multiple,
caninomauro2026multiplicity,mauro2026multiplicity,squassina2009existence}.
However, these works focus mainly on the multiplicity of solutions, typically
by means of an Equivariant Mountain Pass Theorem, but do not address the
existence of least energy solutions or the competitive case $\beta<0$. In
contrast, in this paper we study least energy solutions in both the
cooperative and competitive regimes.

Problem \eqref{Qb} has a formal variational structure with associated energy
functional $\Jb:H_0^1(\Omega;\R^2)\to\R$ defined by
\[
\begin{aligned}
\Jb(u)&=\sum_{k=1}^2\left[\frac12\int_\Omega 
A_k(x,u_k)\nabla u_k\cdot\nabla u_k 
-\frac{\lambda_k}{2}\int_\Omega u_k^2 
-\frac1p\int_\Omega |u_k|^p\right]
\\
&\quad-\frac{2\beta}{p}\int_\Omega
|u_1|^{\frac p2}|u_2|^{\frac p2}.
\end{aligned}
\]
This functional is continuous on $H_0^1(\Omega;\R^2)$ but, in general, is not
differentiable on the whole energy space. For every direction
$v=(v_1,v_2)\in C_c^\infty(\Omega;\R^2)$, its directional derivative is given
by
\[
\begin{aligned}
\langle\Jb'(u),v\rangle
=\sum_{i=1}^2\bigg[&
\int_\Omega A_i(x,u_i)\nabla u_i\cdot\nabla v_i
+\frac12\int_\Omega
\left(D_{u_i}A_i(x,u_i)\nabla u_i\cdot\nabla u_i\right)v_i\\
&-\lambda_i\int_\Omega u_i v_i
-\int_\Omega g_{\beta,i}(u)v_i\bigg].
\end{aligned}
\]

This leads to the following notion of solution:

\begin{definition}
A function $u\in H_0^1(\Omega;\R^2)$ is a weak solution of \eqref{Qb} if
\[
\begin{aligned}
&\int_\Omega A_i(x,u_i)\,\nabla u_i\cdot\nabla v_i
+\frac12\int_\Omega
\left(D_{u_i}A_i(x,u_i)\,\nabla u_i\cdot\nabla u_i\right)v_i\\
&\qquad=\lambda_i\int_\Omega u_i v_i
+\int_\Omega g_{\beta,i}(u)\,v_i,
\end{aligned}
\]
for every $v=(v_1,v_2)\in C_c^\infty(\Omega;\R^2)$ and every $i=1,2$.
\end{definition}

In general, however, one cannot take an arbitrary
$v\in H_0^1(\Omega;\R^2)$, since the term
\[
\int_\Omega
\left(D_{u_i}A_i(x,u_i)\,\nabla u_i\cdot\nabla u_i\right)v_i
\]
need not be well-defined when $N\ge2$.
In view of these difficulties, and in particular of the lack of smoothness of
$\Jb$, we prove the existence of weak solutions to \eqref{Qb} for
$\lambda_1,\lambda_2<\mu_1$ and $\beta\in\R$ by means of variational methods
based on the weak slope and the nonsmooth critical point theory developed in
\cite{nonsmooththeory1,corvellec1993deformation,degiovanni1994critical}.

We also note that if $u_j\equiv0$ for some $j\in\{1,2\}$, then $u_i$, with
$i\ne j$, is a weak solution of the scalar problem
{\footnotesize
\begin{equation}\tag{$Q_i$}\label{Q_i}
\begin{cases}
-\text{div}(A_i(x,u_i)\nabla u_i)+\frac12 D_{u_i}A_i(x,u_i)\nabla u_i\cdot\nabla u_i=\lambda_i u_i+|u_i|^{p-2}u_i&\text{in $\Omega$}\\
u_i=0&\text{on $\partial\Omega$}.
\end{cases}
\end{equation}
}
This type of solution is called a semi-trivial solution, and we denote by
$L_i$ the least energy level of \eqref{Q_i}. The existence of a least energy
solution of \eqref{Q_i} is investigated in Section \ref{section 4}.

The main aim of this paper is to find a fully nontrivial solution
$u=(u_1,u_2)$, i.e., $u_1,u_2\not\equiv0$, and to investigate the minimizers
of
{\small
\begin{equation}\label{eb}
    e_\beta:=\inf\left\{\Jb(u_1,u_2):\ u_1,u_2\not\equiv0,\text{$u=(u_1,u_2)\in H_0^1(\Omega;\R^2)$ solves \eqref{Qb}}\right\},
\end{equation}
}
namely, the least energy solutions of \eqref{Qb}.

We now state our main results, beginning with the cooperative case $\beta>0$.
To treat this regime, we adapt the approach of \cite{clapp2022solutions} to
the nonsmooth setting. More precisely, a minimax argument yields arbitrarily
many fully nontrivial solutions with energy level $c<\min\{L_1,L_2\}$ when
$\beta$ is sufficiently large. We then prove the existence of a least energy
solution by showing that the level $e_\beta$ is achieved.

\begin{theorem}\label{thm1.2}
    Assume that \eqref{a.0}-\eqref{a.2} hold, $A_i(x,-s)=A_i(x,s)$ and $\lambda_1,\lambda_2<\frac{p-2-\gamma}{p-2}\nu\mu_1$. Then, for every $k\in\N$ there exists $\beta_k>0$ such that the problem \eqref{Qb} has at least $k$ fully nontrivial weak solutions for every $\beta>\beta_k$. Furthermore, for such $\beta$, $e_\beta$ is achieved by some $u\in H_0^1(\Omega;\R^2)$, i.e. $u$ is a least energy solution of \eqref{Qb}.
\end{theorem}

We next consider the competitive case $\beta<0$, which is more delicate
because a critical point obtained by a linking argument may be semi-trivial.
We therefore minimize the energy functional on the set
 \[
 \resizebox{\linewidth}{!}{$\displaystyle
 \Nb:=\left\{ u\in \mathcal H:\ 
 \begin{aligned}
 \int_\Omega A_i(x,u_i)\nabla u_i\cdot\nabla u_i
 &+\frac12\int_\Omega
 \left(D_{u_i}A_i(x,u_i)\nabla u_i\cdot\nabla u_i\right)u_i\\
 &-\lambda_i\int_\Omega u_i^2
 =\int_\Omega |u_i|^p
 +\beta\int_\Omega |u_1|^{\frac p2}|u_2|^{\frac p2},
 \quad i=1,2
    \end{aligned} \right\}
 $}
 \]
where 
\[\mathcal H:=\left\{u=(u_1,u_2)\in H_0^1(\Omega;\R^2):\ u_1,u_2\not\equiv0\right\}.\]

Since $\Jb$ is not differentiable, the Nehari set $\Nb$ is not a
$C^1$-manifold, so standard minimization techniques on smooth manifolds do
not apply. Following \cite{clapp2019simple}, we overcome this difficulty by
constructing a homeomorphism between $\Nb$ and
$\partial B_1\times\partial B_1$, which reduces the minimization problem to a
more tractable one. This leads to our main result in the competitive regime.

\begin{theorem}\label{main thm competitive quasilinear}
    Assume that \eqref{a.0}-\eqref{a.3} hold, $A_i(x,-s)=A_i(x,s)$ and let $\beta<-1$. The system  \eqref{Qb}
has a least energy solution $u\in H_0^1(\Omega;\R^2)$ for every $\lambda_1,\lambda_2<\frac{p-2-\gamma}{p-2}\nu\mu_1$, which is also non-negative.
\end{theorem}



\section{Preliminaries}

We recall some results about the critical point theory of continuous functionals, developed in \cite{nonsmooththeory1}. In this setting, we consider $(X,d)$ a metric space and $f: X \to \R$ a continuous functional.

\begin{definition} \label{definition 1}
Let $X$ be a metric space and let $f: X \to \mathds{R}$ be a continuous function. We consider $\sigma \ge 0$ such that there exist $\delta > 0$ and a continuous map $\mathscr{H}: B_{\delta}(u) \times [0, \delta] \to X$ such that
\begin{align}
\label{condition 1}
d(\mathscr{H}(v, t), v) \le t, \\
\label{condition 2}
f(\mathscr{H}(v, t)) \le f(v) - \sigma t.
\end{align}
We define 
\[
|df|(u):=\sup\left\{
\begin{aligned}
 &\text{there exist $\delta>0$ and}\\
\sigma\ge0:\ \ \ \  &\text{$\mathscr H\in C(B_\delta(u)\times[0,\delta];X)$}\\
&\text{which satisfy \eqref{condition 1} and \eqref{condition 2}}
\end{aligned}
\right\}
\]
as the weak slope of $f$ at $u$.
\end{definition}

By the definition of weak slope, it follows that $|df|(u)\le|\nabla f|(u)$, where
\[
|\nabla f|(u)=
\begin{cases}
\displaystyle{\limsup_{v\to u}}\frac{f(u)-f(v)}{d(u,v)}&\text{if $u$ is not a local minimum},\\
0&\text{if $u$ is a local minimum}
\end{cases}
\]
is the (strong) slope introduced in \cite{de1980problemi}, see also \cite[Definition (2.8), p. 80]{degiovanni1994critical}.

\begin{theorem}[{\cite[Theorem 1.1.2]{nonsmooththeory1}}]\label{theorem 1}
Let $E$ be a normed space and $X \subset E$ an open subset. Let us fix $u \in X$ and $v \in E$ with $\|v\| = 1$. For each $w \in X$ we define
$$\overline{D}_+f(w)[v]:= \limsup_{t \to 0^+} \frac{f(w + tv) - f(w)}{t}.$$
Then $|df|(u) \ge -\limsup_{w \to u} \overline{D}_+f(w)[v]$.
\end{theorem}
In particular, one can deduce that $|df|(u)=\|f'(u)\|$ for every differentiable function defined on a Banach space, see \cite[Corollary 1.1.3]{nonsmooththeory1}.

\begin{definition}\label{def p.critical}
Let $X$ be a metric space and let $f: X \to \mathds{R}$ be continuous. We will say that $u \in X$ is a (lower) critical point if $|df|(u) = 0$. A (lower) critical point is said at the level $c \in \mathbb{R}$ if it is also true that $f(u) = c$.
\end{definition}

\begin{definition}\label{ps}
Let $X$ be a metric space and let $f: X \to \mathds{R}$ be continuous. A sequence $\{u_n\} \subset X$ is said a $(PS)_c$-sequence if
\begin{align}
\label{ps1}
&f(u_n) \to c, \\
\label{ps2}
&|df|(u_n) \to 0.
\end{align}
Furthermore, we will say that $f$ satisfies the $(PS)_c$-condition if every  $(PS)_c$-sequence admits a convergent subsequence in $X$. 
If the $(PS)_c$-condition holds for every $c \in \mathds{R}$ we write $(PS)$-condition to simplify the notation.
\end{definition}
\begin{theorem}[Equivariant Mountain Pass, {\cite[Theorem 1.3.3]{nonsmooththeory1}}]\label{MPequi}	\hfill\\
Let $X$ be a Banach space and let $f:X\to\R$ be a continuous even functional. Suppose that
\begin{itemize}
\item $\exists\ \rho>0,\alpha>f(0)$ and a subspace $X_2\subset X$ of finite codimension such that $f\ge\alpha$ on $\partial B_{\rho}\cap X_2$,
\item for every subspace $X_1^{(k)}$ of dimension $k$, there exists $R=R^{(k)}>0$ such that $f\le f(0)$ in $B_R^c\cap X_1^{(k)}$.
\end{itemize}
If $f$ satisfies the $(PS)_c$-condition for every $c\ge\alpha$, then there exists  a divergent sequence of critical values.
\end{theorem}

We  also recall the Linking Theorem proved by Bartolo, Benci and Fortunato in \cite{bartolo1983abstract}, but for continuous functionals:

\begin{theorem}\label{linking theorem}
    Let $X$ be a Banach space and assume that $f:X\to\R$ is an even continuous function such that $f(0)=0$.\\ Suppose also there exist closed subspaces $Y,Z\subset X$ and constants $\alpha, c_0,\rho$ such that $Y$ has finite codimension, $Z$ has finite dimension, $0<\alpha<c_0, \rho>0$ and 
    {\small
    \begin{align*}
        &\inf_{x\in Y\cap\partial B_\rho}f(x)>\alpha, &\max_{x\in Z} f(x)<c_0.
    \end{align*}}
    If $f$ satisfies the Palais-Smale condition at level $c\in(\alpha,c_0)$ and if $\dim Z-\text{codim } Y>0$, then $f$ has either $k$ critical values in $(\alpha,c_0)$ or it has infinitely many critical points in $f^{-1}(\alpha,c_0)$. 
\end{theorem}
\begin{proof}
The proof is the same obtained in the regular case in \cite[Theorem 2.9 and Theorem 2.4]{bartolo1983abstract}, but using the Equivariant Deformation Theorem  \cite[Theorem 1.2.5]{nonsmooththeory1}.
\end{proof}

Having recalled the general abstract tools and critical point
theorems of nonsmooth theory, we now develop the variational framework
associated with the energy functional $\Jb$. This framework will be used to study the
Palais--Smale compactness condition for our problem.

We set
\[Q_i(u_i):=-\text{div}(A_i(x,u_i)\nabla u_i)+\frac12 D_{u_i}A_i(x,u_i)\nabla u_i\cdot\nabla u_i,\quad i=1,2,\]
\[T_i(u_i,\varphi_i):=\int_\Omega A_i(x,u_i)\nabla u_i\cdot\nabla\varphi_i+\frac12\int_\Omega D_{u_i}A_i(x,u_i)\nabla u_i\cdot\nabla u_i\varphi_i,\ \ i=1,2,\]
\[T(u,\varphi):=T_1(u_1,\varphi_1)+T_2(u_2,\varphi_2),\]
for every $\varphi\in C_c^\infty(\Omega;\R^2)$. 

\begin{definition}
A sequence $\{u_n\}\subset H_0^1(\Omega;\R^2)$ is a Concrete Palais-Smale sequence at level $c$ for $\Jb$, $(CPS)_c$-sequence, if
\begin{itemize}
\item $\Jb(u_n)\to c$ in $\R$.
\item $Q_iu_{n,i}-\lambda_iu_{n,i}-g_{\beta,i}(u_n) \in H^{-1}(\Omega)$, eventually as $n\to+\infty$,  for every $i=1,2.$
\item $Q_iu_{n,i}-\lambda_iu_{n,i}-g_{\beta,i}(u_n)\to 0$ strongly in $H^{-1}(\Omega)$ for every $i=1,2.$
\end{itemize}
We will say that $\Jb$ satisfies the $(CPS)_c$-condition if every  $(CPS)_c$-sequence admits a convergent subsequence in $H_0^1(\Omega;\R^2)$.
\end{definition}
\begin{theorem}\label{differentiability Jb}
The functional $\Jb:H_0^1(\Omega;\R^2)\to\R$ is continuous and for every $u\in H_0^1(\Omega;\R^2)$ we have
\[
\resizebox{\linewidth}{!}{$\displaystyle
|d\Jb|(u)\ge\sup_{\substack{\varphi\in C_c^{\infty}(\Omega;\R^2)\\ \|\varphi\|_{H_0^1(\Omega;\R^2)}\le1}}\left(T(u,\varphi)-\int_\Omega(\lambda_1u_1\varphi_1+g_{\beta,1}(u)\varphi_1+\lambda_2u_2\varphi_2+g_{\beta,2}(u)\varphi_2)\right).
$}
\]
Furthermore, if  $|d\Jb|(u)<+\infty$, then $Q_iu_i-\lambda_iu_i-g_{\beta,i}(u)\in H^{-1}(\Omega)$ for every $i=1,2,$ and
\[
\resizebox{\linewidth}{!}{$\displaystyle
\|(Q_1u_1-\lambda_1u_1-g_{\beta,1}(u), Q_2u_2-\lambda_2u_2-g_{\beta,2}(u))\|_{H^{-1}(\Omega;\R^2)}\le|d\Jb|(u).
$}
\]
\end{theorem}
\begin{proof}
The continuity of the functional follows from the fact that $A_i(x,u)$  are continuous in $u_i$ and by the Lebesgue's Theorem. 
Let $\varphi \in C_c^\infty(\Omega;\R^2)$ with $\|\varphi\|\le1$. Then, by the Lebesgue Theorem, the quantity
\begin{align*}
\langle \Jb'(u), \varphi \rangle 
&:= \lim_{t\to0} \frac{\Jb(u+t\varphi) - \Jb(u)}{t} \\
&= T_1(u_1, \varphi_1) + T_2(u_2, \varphi_2)\\
&\quad- \int_\Omega (\lambda_1 u_1 + g_{\beta,1}(u)) \varphi_1- \int_\Omega (\lambda_2 u_2 + g_{\beta,2}(u)) \varphi_2
\end{align*}
is well-defined, and the map
\[
H_0^1(\Omega;\R^2)\ni u\mapsto \langle \Jb'(u), \varphi \rangle 
\]
is continuous for every $\varphi\in C_c^\infty(\Omega;\R^2)$.

Moreover, from Theorem \ref{theorem 1}, we have
\[
\langle \Jb'(u),\varphi\rangle=\limsup_{\substack{w\to u\\\ t\to0^+}}\frac{\Jb(w+t\varphi)-\Jb(w)}{t}\ge-|d\Jb|(u)\cdot\|\varphi\|_{H_0^1},
\]
 and $|d\Jb|(u)\ge\langle \Jb'(u),-\varphi\rangle$. Since $\varphi$ is arbitrary, the conclusion follows.
\end{proof}
\begin{corollary}\label{PS e CPS}
The following facts hold:
\begin{enumerate}
\item[$(i)$] if $u$ is a lower critical point for $\Jb$, then $u$ is a weak solution of  \eqref{Qb},
\item[$(ii)$] every $(PS)_c$-sequence is also a $(CPS)_c$-sequence,
\item[$(iii)$] $\Jb$ satisfies $(CPS)_c$-condition$\implies \Jb$ satisfies $(PS)_c$-condition.
\end{enumerate}
\end{corollary}
We conclude this section with a useful density result and a remark about the regularity:

\begin{theorem}\label{density 2.3.2}
Let $w=(w_1,w_2)\in H^{-1}(\Omega;\R^2)$, $u=(u_1,u_2)\in H_0^1(\Omega;\R^2)$ such that
\[
\begin{cases}
    -\text{div}(A_1(x,u_1)\nabla u_1)+\frac12D_{u_1}A_1(x,u_1)\nabla u_1\cdot\nabla u_1=w_1&\text{in $\mathcal D'(\Omega)$},\\
   -\text{div}(A_2(x,u_2)\nabla u_2)+\frac12D_{u_2}A_2(x,u_2)\nabla u_2\cdot\nabla u_2=w_2&\text{in $\mathcal D'(\Omega)$},
\end{cases}
\]
and let $v=(v_1,v_2)\in H_0^1(\Omega;\R^2)$ such that \[\left[\left(D_{u_i}A_i(x,u_i)\nabla u_i\cdot\nabla u_i\right)v_i\right]^-\in L^1(\Omega),\quad i=1,2.\]

Then 
\[\left(D_{u_1}A_1(x,u_1)\nabla u_1\cdot\nabla u_1\right)v_1, \left(D_{u_2}A_2(x,u_2)\nabla u_2\cdot\nabla u_2\right)v_2\in L^1(\Omega),\]
and
\[T_1(u_1,v_1)=\langle w_1,v_1\rangle,\quad T_2(u_2,v_2)=\langle w_2,v_2\rangle.\]
\end{theorem}

\begin{proof}
One can apply \cite[Theorem 3.1]{caninoserdica} for every $i=1,2$.
\end{proof}
\begin{remark}\label{regularity}
 From \cite[Theorem 4.1]{caninoserdica} and the subcritical growth of \( G_\beta \), see \eqref{g.1}, it follows that any weak solution \( u \in H_0^1(\Omega; \mathbb{R}^2) \) of \eqref{Qb} is bounded, i.e., \( u_i \in H_0^1(\Omega) \cap L^\infty(\Omega) \) for each \( i = 1,2 \).  
Moreover, by Theorem \ref{density 2.3.2}, we can test \eqref{Qb} with any bounded function \( v \in H_0^1(\Omega; \mathbb{R}^2) \cap L^\infty(\Omega; \mathbb{R}^2) \), as well as with any weak solution itself.
\end{remark}

\section{The Palais-Smale condition}

\begin{lemma}\label{PS1}
Let $c\in\R$ be a real number and $\lambda_1,\lambda_2<\frac{p-2-\gamma}{p-2}\nu\mu_1,\beta\in\R$.  Then every $(CPS)_c$-sequence for $\Jb$ is bounded in $H_0^1(\Omega;\R^2)$.
\end{lemma}
\begin{proof}
Let $\{u_n\}\subset H_0^1(\Omega;\R^2)$ be a $(CPS)_c$-subsequence for $\Jb$ and let  \[
w_{n,i}:=Q_iu_{n,i}-\lambda_iu_{n,i}-g_{\beta,i}(u_n),\]
for every $i=1,2$.
By \eqref{a.2} and Theorem \ref{density 2.3.2}, we have:
\begin{align*}
&-\|w_{n,i}\|_{H^{-1}}\cdot\|u_{n,i}\|_{H_0^1}\le\langle w_{n,i},u_{n,i}\rangle=\int_{\Omega}A_i(x,u_{n,i})\nabla u_{n,i}\cdot\nabla u_{n,i}\\
&+\frac12\int_{\Omega}\left(D_{u_{n,i}}A_i(x,u_{n,i})\nabla u_{n,i}\cdot\nabla u_{n,i}\right)u_{n,i}-\lambda_i\int_{\Omega}u_{n,i}^2-\int_{\Omega}g_{\beta,i}(u_n)u_{n,i}.
\end{align*}
According to \eqref{a.2}:
\begin{align*}
\int_{\Omega}\left(D_{u_{n,i}}A_i(x,u_{n,i})\nabla u_{n,i}\cdot\nabla u_{n,i}\right)u_{n,i}\le\gamma\int_{\Omega}A_i(x,u_{n,i})\nabla u_{n,i}\cdot\nabla u_{n,i}.
\end{align*}
%
Thus, setting $w_n=(w_{n,1},w_{n,2})\in H^{-1}(\Omega;\R^2),$ we get
\begin{equation}\label{limitatezza PS}
\begin{aligned}
    C+1+\|u_n\|&\ge \Jb(u_n)-\frac{1}{p}\langle w_n,u_n\rangle\\
    &=\frac{p-2}{2p}\sum_{i=1}^2
    \int_{\Omega}A_i(x,u_{n,i})\nabla u_{n,i}\cdot\nabla u_{n,i}\\
    &\quad-\frac{1}{2p}\int_{\Omega}\left(D_{u_{n,1}}A_1(x,u_{n,1})\nabla u_{n,1}\cdot\nabla u_{n,1}\right) u_{n,1}\\
    &\quad-\frac{1}{2p}\int_{\Omega}\left(D_{u_{n,2}}A_2(x,u_{n,2})\nabla u_{n,2}\cdot\nabla u_{n,2}\right)u_{n,2}\\
    &\quad-\frac{p-2}{2p}\left(\lambda_1\int_{\Omega}u_{n,1}^2+\lambda_2\int_{\Omega}u_{n,2}^2\right)\\
    &\ge\frac{p-2-\gamma}{2p}\sum_{i=1}^2\biggl(\int_{\Omega}A_i(x,u_{n,i})\nabla u_{n,i}\cdot\nabla u_{n,i}\biggr)\\
    &\quad-\frac{p-2}{2p}\left(\lambda_1\int_{\Omega}u_{n,1}^2+\lambda_2\int_{\Omega}u_{n,2}^2\right).
\end{aligned}
\end{equation}
According to \eqref{a.1}, we have
\[
\begin{aligned}
C+1+\|u_n\|\ge\sum_{i=1}^2\bigg[&
\frac{p-2-\gamma}{2p}\nu\int_\Omega|\nabla u_{n,i}|^2-\frac{p-2}{2p}\lambda_i\int_\Omega u_{n,i}^2\bigg].
\end{aligned}
\]
If $\lambda_1,\lambda_2<\frac{p-2-\gamma}{p-2}\nu\mu_1$ we have there exists a constant $C_1>0$ such that
\begin{align*}
    C+1+\|u_n\|\ge C_1(\|u_{n,1}\|_{H_0^1}^2+\|u_{n,2}\|_{H_0^1}^2)
\end{align*}
and this implies that $\{u_n\}$ is bounded in $H_0^1(\Omega;\R^2)$. \end{proof}
\begin{lemma}\label{remark PS}
Let $c\in\R$, $\beta\ge0$, and $\lambda_1,\lambda_2\in\R$. Then every $(CPS)_c$-sequence for $\Jb$ is bounded in $H_0^1(\Omega;\R^2)$.
\end{lemma}
\begin{proof}
From \eqref{limitatezza PS}, we have
\[
\begin{aligned}
\frac{p-2-\gamma}{2p}\sum_{i=1}^2
\int_{\Omega}A_i(x,u_{n,i})\nabla u_{n,i}\cdot\nabla u_{n,i}
&\le\frac{p-2}{2p}\sum_{i=1}^2\lambda_i\int_\Omega u_{n,i}^2+C(1+\|u_n\|).
\end{aligned}
\]
Suppose that $\|u_n\|\to+\infty$ as $n\to+\infty$. Thus,
\[
C+1+\|u_n\|\le\frac{p-2-\gamma}{4p}\nu(\|u_{n,1}\|^2+\|u_{n,2}\|^2),
\]
and there exists $c_1>0$ such that
\[
\|u_{n,1}\|^2+\|u_{n,2}\|^2\le c_1\sum_{i=1}^2\max\{\lambda_i,0\}\int_\Omega u_{n,i}^2.
\]
If $\beta\ge0$, then we obtain 
\[
\resizebox{\linewidth}{!}{$\displaystyle
\frac{\Jb(u_{n,1},u_{n,2})}{\|u_{n,1}\|^2+\|u_{n,2}\|^2}=\frac12\frac{\displaystyle\sum_{i=1}^2\biggl(\int_{\Omega}A_i(x,u_{n,i})\nabla u_{n,i}\cdot\nabla u_{n,i}-\lambda_i\int_\Omega u_{n,i}^2\biggr)}{\|u_{n,1}\|^2+\|u_{n,2}\|^2}-\frac{\displaystyle\int_\Omega G_\beta(u_{n,1},u_{n,2})}{\|u_{n,1}\|^2+\|u_{n,2}\|^2},
$}
\]
and the right-hand side goes to $-\infty$ when $n\to+\infty$, which is a contradiction. Notice that the case $\beta=0$ is included, since
\[
G_0(u_1,u_2)=\frac1p(|u_1|^p+|u_2|^p)\ge0.
\]
Thus, if $\beta\ge0$ and $\lambda_1,\lambda_2\in\R$, then every $(CPS)_c$-sequence for $\Jb$ is bounded in $H_0^1(\Omega;\R^2)$.
\end{proof}

\begin{theorem}\label{PS and boundeness}
For every $c\in\R$ and for any $\lambda_1,\lambda_2,\beta\in\R$ the following facts are equivalent:
\begin{enumerate}
\item[$(i)$] $\Jb$ satisfies the $(CPS)_c$-condition.
\item[$(ii)$] Every $(CPS)_c$-sequence for $\Jb$ is bounded in $H_0^1(\Omega;\R^2)$.
\end{enumerate}
\end{theorem}
\begin{proof}
$(i)\implies(ii)$. Let $\{u_n\}$ be a $(CPS)_c$-sequence for $\Jb$. Since $\Jb$ satisfies the $(CPS)_c$-condition, we have $u_n\to u\in H_0^1(\Omega;\R^2)$ up to a subsequence. In particular, $\{u_n\}$ is bounded in $H_0^1(\Omega;\R^2)$, namely $(ii)$.\\ 
$(ii)\implies(i)$. Let $\{u_n\}$ be a bounded $(CPS)_c$-sequence for $\Jb$. Since $\{u_n\}$ is bounded, $u_n\wto u$ in $H_0^1(\Omega;\R^2)$. Furthermore,
\[\lambda_i\int_\Omega u_{n,i}\varphi_i\to\lambda_i\int_\Omega u_i\varphi_i,\quad \int_\Omega g_{\beta,i}(u_n)\varphi_i\to\int_\Omega g_{\beta,i}(u)\varphi_i\]
for all $\varphi=(\varphi_1,\varphi_2)\in H_0^1(\Omega;\R^2)$ and for every $i=1,2$, by the Sobolev embedding. Clearly,
$$w_{n,i}:=Q_iu_{n,i}-\lambda_iu_{n,i}-g_{\beta,i}(u_n)\to0$$
in $H^{-1}(\Omega)$ for every $i=1,2$. Moreover,
\[
Q_i u_{n,i}=\lambda_i u_{n,i}+g_{\beta,i}(u_n)+w_{n,i},
\]
and the right-hand side converges in $H^{-1}(\Omega)$. Hence, by \cite[Lemma 3.4]{caninoserdica}, $u_n\to u$ strongly in $H_0^1(\Omega;\mathbb R^2)$ up to a subsequence.
\end{proof}
Since we have already proven that any $(CPS)_c$-sequence  is bounded in Lemma \ref{PS1}, we obtain the following:
\begin{corollary}\label{PS condition}
    The functional $\Jb$ satisfies the $(PS)_c-$condition, for every $\lambda_1,\lambda_2\in\R$ and $\beta\ge0$.
\end{corollary}
\begin{proof}
    Let \( \{u_n\} \) be a \((PS)_c\)-sequence. Then \( \{u_n\} \) is also a \((CPS)_c\)-sequence by Corollary~\ref{PS e CPS}-$(ii)$. By Lemma~\ref{PS1}, Lemma \ref{remark PS} and Theorem~\ref{PS and boundeness}, it follows that \( \Jb \) satisfies the \((CPS)_c\)-condition. The claim then follows from Corollary~\ref{PS e CPS}-$(iii)$.
\end{proof}
\section{The scalar problem}\label{section 4}
We consider the scalar equation \eqref{Q_i} for every $i=1,2$ and we define the associated energy functional 
\[\I(u_i)=\frac12\int_\Omega A_i(x,u_i)\nabla u_i\cdot\nabla u_i-\frac{\lambda_i}{2}\int_\Omega u_i^2-\frac{1}{p}\int_\Omega |u_i|^p.\]
\begin{lemma}
 Let $i=1,2$. We have
    \begin{itemize}
        \item[$(i)$] There exists a subspace $X\subseteq H_0^1(\Omega)$ with finite codimension and $\rho,\alpha>0$ such that $\I\ge\alpha$ on $\partial B_\rho\cap X$.
        \item[$(ii)$] For every subspace $V\subset H_0^1(\Omega)$ with finite dimension, there exists $R>0$ such that $\I\le0$ in $B_R^c\cap V$.
    \end{itemize}
\end{lemma}

\begin{proof}
\begin{itemize}
    \item[$(i)$] By \eqref{a.1}, we have
    \[\I(u_i)\ge\frac{\nu}{2}\int_\Omega|\nabla u_i|^2-\frac{\lambda_i}{2}\int_\Omega u_i^2-\frac{1}{p}\int_\Omega |u_i|^p.\]
    Let $\phi_k$ be an eigenfunction associated to the $k$-th eigenvalue $\mu_k$ of the Laplacian with Dirichlet boundary conditions and we define 
    \[X_{K}=\overline{\text{span}\{\phi_k: k\ge K\}}.\]
    For each $\rho>0$ and $z\in\partial B_\rho\cap X_K$, we obtain:
    \[
    \begin{aligned}
       \I(z)&\ge\frac{\nu}{2}\int_\Omega|\nabla  z|^2-\frac{\lambda_i}{2}\int_\Omega z^2-\frac{1}{p}\int_\Omega|z|^p\\
       &\ge\frac12\left(\nu-\frac{\lambda_i}{\mu_K}\right)\|z\|^2-\frac{C}{p}\|z\|^p\\
       &=\frac12\left(\nu-\frac{\lambda_i}{\mu_K}\right)\rho^2-\frac{C}{p}\rho^p.
    \end{aligned}
    \]
    Now, we can choose $K\ge1$ large enough and $\rho>0$ small enough such that $\I(z)\ge\alpha$, for some $\alpha>0$.
    \item[$(ii)$] Let $e\in V$ with $\|e\|=1$ and we set $v=Re$. Then
    \begin{align*}
        \I(Re)=\frac{R^2}{2}\int_\Omega A_i(x,Re)\nabla e\cdot\nabla e-\frac{\lambda_iR^2}{2}\int_\Omega e^2-\frac{R^p}{p}\int_\Omega |e|^p.
    \end{align*}
    Since the coefficients of $A_i(x,s)$ are bounded functions and $p>2$, we get $\I(Re)\to-\infty$ as $R\to+\infty$ and the proof is complete.\qedhere
\end{itemize}
\end{proof}

\begin{theorem}
    The functional $\I$ satisfies the $(PS)$-condition for every $i=1,2$.
\end{theorem}
\begin{proof}
Let $\{u_{n,1}\}$ be a $(PS)$-sequence for $\mathcal I_{\lambda_1}$. Then $\{(u_{n,1},0)\}$ is a $(CPS)$-sequence for $\mathcal J_{\lambda,0}$. According to Corollary \ref{PS condition}, $\{u_{n,1}\}$ admits a convergent subsequence in $H_0^1(\Omega)$; hence $\mathcal I_{\lambda_1}$ satisfies the $(PS)$-condition. The case $i=2$ is similar.
\end{proof}

\begin{theorem}\label{multiplicity scalar}
    Assume that \eqref{a.0}-\eqref{a.2} hold and that $A_i(x,-s)=A_i(x,s)$ for every $i=1,2$. Then there exist infinitely many weak solutions of \eqref{Q_i} for every $i=1,2$. Furthermore, any weak solution $u_i$ is bounded, i.e. $u_i\in H_0^1(\Omega)\cap L^\infty(\Omega)$.
\end{theorem}
\begin{proof}
    The existence follows by Theorem \ref{MPequi}, as in \cite[Theorem 2.6]{caninoquasilineare1}. The regularity is obtained by applying \cite[Theorem 2.2.5-$(b)$]{nonsmooththeory1}.
\end{proof}

Now, we define the least energy level of \eqref{Q_i} 
\[L_i=\inf\left\{\I(z):\ z\in H_0^1(\Omega)\setminus\{0\},\ \text{$z$ solves \eqref{Q_i}}\right\}.\]
This infimum is well-defined because the set of nontrivial weak solution of \eqref{Q_i} is non-empty by Theorem \ref{multiplicity scalar}. We want to show that $L_i$ is achieved by some $z\in H_0^1(\Omega)\setminus\{0\}$.
\begin{theorem}\label{least energy scalar}
 Let $i=1,2$ and assume that $\lambda_i<\frac{p-2-\gamma}{p-2}\nu\mu_1$.
    There exists $u_i\in H_0^1(\Omega)\setminus\{0\}$ a weak solution of \eqref{Q_i} such that $\I(u_i)=L_i>0$, i.e. $u_i$ is a least energy solution of \eqref{Q_i}.
\end{theorem}
\begin{proof}
    Let $\{z_n\}\subset H_0^1(\Omega)\setminus\{0\}$ be a sequence such that $z_n$ solves \eqref{Q_i} for every $n\ge1$, i.e. $|d\I|(z_n)=0$ for every $n\ge1$. By the $(PS)$-condition, we can assume that $z_n\to z$ with $z\in H_0^1(\Omega)$, up to a subsequence.
    According to \cite[Theorem 4.1]{caninoserdica}, we have $z_n\in L^\infty(\Omega)$ and 
    {\small
    \begin{equation}\label{eq nehari scalare}
        \begin{aligned}
          0&=\langle\I'(z_n),z_n\rangle=\int_\Omega A_i(x,z_n)\nabla z_n\cdot\nabla z_n+\frac12\int_\Omega (D_{z_n}A_i(x,z_n)\nabla z_n\cdot\nabla z_n)z_n\\
        &\quad-\lambda_i\int_\Omega z_n^2-\int_\Omega|z_n|^p.  
        \end{aligned}
    \end{equation}
    }
    According to \eqref{a.0}-\eqref{a.2} and by Sobolev's embedding, we get
    \[0\ge\nu\|z_n\|^2-\lambda_i\int_\Omega z_n^2-C\|z_n\|^p.\]
    Since $\lambda_i<\nu\mu_1$,
    \[\sqrt{\nu\int_\Omega |\nabla z_n|^2-\lambda_i\int_\Omega z_n^2}\]
    is an equivalent norm for $H_0^1(\Omega)$ and  there exists $\theta_i=\theta_i(\nu,\lambda_i)>0$ such that
    \[C\|z_n\|^p\ge\theta_i\|z_n\|^2.\]
    Hence, $z\not\equiv0$ and we have $\I(z)=L_i$. It remains to prove that $L_i>0$. Indeed, by \eqref{eq nehari scalare} and \eqref{a.2}:
    \begin{align*}
        \I(z)&=\left(\frac12-\frac1p\right)\int_\Omega
        \left[A_i(x,z)\nabla z\cdot\nabla z-\lambda_iz^2\right]\\
        &\quad-\frac{1}{2p}\int_\Omega
        \left(D_zA_i(x,z)\nabla z\cdot\nabla z\right)z\\
        &\ge \left(\frac12-\frac1p\right)\int_\Omega
        \left[A_i(x,z)\nabla z\cdot\nabla z-\lambda_iz^2\right]\\
        &\quad-\frac{\gamma}{2p}\int_\Omega A_i(x,z)\nabla z\cdot\nabla z\\
        &\ge\frac{p-2-\gamma}{2p}\nu\int_\Omega|\nabla z|^2-\frac{p-2}{2p}\lambda_i\int_\Omega z^2\\
        &\ge\left(\frac{p-2-\gamma}{2p}\nu-\frac{(p-2)\lambda_i}{2p\mu_1}\right)\|z\|^2>0.
    \end{align*}
    Thus, $L_i=\I(z)>0$ by the assumption $\lambda_i<\frac{p-2-\gamma}{p-2}\nu\mu_1$.
\end{proof}

\section{Proof of Theorem \ref{thm1.2}}
We define the subspace of codimension $k_0=k_{0,1}+k_{0,2}\ge1$
{\footnotesize
\[
Y_{k_0}:=\overline{\left\{u=(u_1,u_2)\in H_0^1(\Omega;\R^2):\quad u_1=\sum_{n\ge k_{0,1}}c_{n,1}\phi_n,\quad u_2=\sum_{n\ge k_{0,2}}c_{n,2}\phi_n\right\}}.
\]
}
  where $\phi_j$ is an eigenfunction associated to eigenvalue problem:
\[
\begin{cases}
    -\Delta \phi_j=\mu_j\phi_j&\text{in $\Omega$},\\
    \phi_j=0&\text{on $\partial\Omega$},
\end{cases}
\]
and $\mu_j\to+\infty$ as $j\to+\infty$.

The functional $\Jb$ has a linking geometry given by the following lemmas. \begin{lemma}\label{geometria1}
     Let $\alpha>0$ and $\beta>0$. We can find $k_0\ge1$ large enough and a suitable radius $\rho>0$  such that $\Jb\ge\alpha$ on $\partial B_\rho\cap Y_{k_0}$.

  \end{lemma}
  \begin{proof}
For each $u\in \partial B_\rho\cap Y_{k_0},$ by Young's inequality:
\begin{align*}
    \Jb(u)&\ge \frac{\nu}{2}\int_\Omega (|\nabla u_1|^2+|\nabla u_2|^2)-\frac{\lambda_1}{2}\int_\Omega u_{1}^2-\frac{\lambda_2}{2}\int_\Omega u_2^2\\
    &\quad-\frac1p\int_\Omega|u_1|^p-\frac{\beta}{p}\int_\Omega |u_1|^p-\frac{\beta}{p}\int_\Omega |u_2|^p-\frac{1}{p}\int_\Omega |u_2|^p\\
    &\ge\frac{\nu}{2}\|u\|^2-\frac{\lambda_1}{2}\int_\Omega u_{1}^2-\frac{\lambda_2}{2}\int_\Omega u_2^2
   -\frac{\beta+1}{p}C(\|u_1\|^p+\|u_2\|^p),
\end{align*}
where $C>0$ is the constant of the embedding $H_0^1(\Omega)\hookrightarrow L^p(\Omega)$. Moreover,
since $u\in Y_{k_0}$, we obtain that
\begin{align*}
\Jb(u)&\ge\frac12\left(\nu-\frac{\lambda_1}{\mu_{k_0,1}}\right)\|u_1\|^2
+\frac12\left(\nu-\frac{\lambda_2}{\mu_{k_0,2}}\right)\|u_2\|^2\\
&\quad-\frac{C(\beta+1)}{p}(\|u_1\|^p+\|u_2\|^p).
\end{align*}
 Thus, we can choose $\rho$ and $k_0=k_{0,1}+k_{0,2}\ge1$ such that $\Jb(u)\ge\alpha$ on $\partial B_\rho\cap Y_{k_0}$.
\end{proof}

\begin{lemma}\label{geometria2}
 Assume that $\lambda_1,\lambda_2<\frac{p-2-\gamma}{p-2}\nu\mu_1$ and let $L=\min\{L_1,L_2\}$, where $L_i$ denotes the least energy level of \eqref{Q_i}. For every $m\in\N$, there exist an $m$-dimensional subspace $Z_m\subset H_0^1(\Omega;\R^2)$ and $\beta_m>0$, such that
\[
\max_{u\in Z_m}\Jb(u)<L
\quad\text{for every }\beta>\beta_m.
\]
\end{lemma}
\begin{proof}
    Let $W_m=span\{\phi_1,\dots,\phi_m\}$ and, for every $w\in W_m$, we compute:

    {\small
    \begin{align*}
        \Jb(w,w)&\le NC_0\left(\int_\Omega|\nabla w|^2-\frac{\lambda_1+\lambda_2}{2NC_0}\int_\Omega w^2\right)-\frac{2}{p}\int_\Omega |w|^p-\frac{2\beta}{p}\int_\Omega |w|^p.
    \end{align*}}
    If $\mu_m\le\frac{\lambda_1+\lambda_2}{2NC_0}$ the conclusion holds for $\beta_m=0$. Otherwise, we take $w_\beta\in W_m\cap B_r$, with $r>0$ independent of $\beta$,  such that $$0<\max\left\{\Jb(w,w):\ w\in W_m\right\}=\Jb(w_\beta,w_\beta).$$
    Since $w_\beta\in H_0^1(\Omega)\cap L^\infty(\Omega)$, we have:
    \begin{align*}
        0&=\langle \Jb'(w_\beta,w_\beta),(w_\beta,w_\beta)\rangle\\
        &\le\int_\Omega A_1(x,w_\beta)\nabla w_\beta\cdot\nabla w_\beta+\int_\Omega A_2(x,w_\beta)\nabla w_\beta\cdot\nabla w_\beta\\
        &\quad+\frac12\int_\Omega \left(D_{w_\beta}A_1(x,w_\beta)\nabla w_\beta\cdot\nabla w_\beta \right)w_\beta\\
        &\quad+\frac12\int_\Omega \left(D_{w_\beta}A_2(x,w_\beta)\nabla w_\beta\cdot\nabla w_\beta\right) w_\beta\\
        &\quad-(\lambda_1+\lambda_2)\int_\Omega w_\beta^2
        -2(1+\beta)\int_\Omega |w_\beta|^{p}.
    \end{align*}
    It follows that there exists $K>0$ independent of $\beta$ such that
    $$\int_\Omega|w_\beta|^{p}\le\frac{K}{\beta},$$
    and this implies that $w_\beta\to0$ in $H_0^1(\Omega)$ as $\beta\to+\infty$, since any norm is equivalent in $W_m$. Hence, there exists $\beta_m>0$ such that, for every $\beta>\beta_m$, we have
\[
\max_{u\in Z_m}\Jb(u)<L,
\]
with $Z_m:=\{(w,w):\ w\in W_m\}$.
\end{proof}


\begin{remark}
    We point out that the restriction
    $\lambda_1,\lambda_2<\frac{p-2-\gamma}{p-2}\nu\mu_1$ is used only to
    ensure that, for each $i=1,2$, the scalar least energy level $L_i$ is
    positive and is achieved by a solution of \eqref{Q_i}. Consequently, if
    one could prove that $L_i>0$ is achieved for every $\lambda_i\in\R$,
    then Lemma \ref{geometria2} would remain valid for every
    $\lambda_1,\lambda_2\in\R$.
\end{remark}
\begin{proof}[Proof of {Theorem \ref{thm1.2}}]
We divide the proof in two steps.\\
\textbf{Step 1}. Existence of arbitrarily many fully nontrivial solutions.\\ 
Let  $L=\min\{L_1,L_2\}$, and let $k_0\ge1$ be the integer found in Lemma \ref{geometria1}. For any $k\ge1$, we choose $m:=k+k_0$. By Lemmas \ref{geometria1}-\ref{geometria2}, for every $\beta>\beta_m$ we have
\[
\dim Z_m-\operatorname{codim}Y_{k_0}=m-k_0=k.
\]
Therefore, Theorem \ref{linking theorem} gives at least $k$ critical points for the functional $\Jb$, $u^{(h)}=(u_1^{(h)},u_2^{(h)})$, with $\Jb(u^{(h)})<\min\{L_1,L_2\}$ and $h=1,\dots,k$. Hence, $u_1^{(h)},u_2^{(h)}\not\equiv0$ and, according to Remark \ref{regularity}, $u^{(h)}\in H_0^1(\Omega;\R^2)\cap L^\infty(\Omega;\R^2)$. \\
\textbf{Step 2}. Existence of a least energy solution.

 Let $\{u_n\}\subset H_0^1(\Omega;\R^2)\cap L^\infty(\Omega;\R^2)$ such that $u_{n,1},u_{n,2}\not\equiv0$, $u_n$ solves \eqref{Qb} and $\Jb(u_n)\to e_\beta$.  Reasoning as in Theorem \ref{least energy scalar} we obtain that $u_n\to u^*\not\equiv0$ and $u^*$ satisfies $\Jb(u^*)< \min\{L_1,L_2\}$ by Step 1. Then $u_1^*,u_2^*\not\equiv0$ and $u^*$ is a least energy solution, i.e. achieves $e_\beta$. 
\end{proof}
\section{Proof of Theorem \ref{main thm competitive quasilinear}}
Throughout this section we assume that $\beta<-1$ and $\lambda_1,\lambda_2<\frac{p-2-\gamma}{p-2}\nu\mu_1$.
Let $l_{\lambda,\beta}:=\inf_{\Nb}\Jb$.
We note that any weak solution, i.e. $u\in H_0^1(\Omega;\R^2)$ such that $|d\Jb|(u)=0$, belongs to $\Nb$. Then $l_{\lambda,\beta}\le e_\beta$ and a critical point at level $l_{\lambda,\beta}$ is also a least energy solution.

For every $u=(u_1,u_2)\in\mathcal H$, we define $h_u:(0,+\infty)^2\to\R$ as follows:
\begin{align*}
  h_u(t)=\Jb(t_1u_1,t_2u_2) &=\sum_{i=1}^2\frac{t_i^2}{2}\int_\Omega\big[ A_i(x,t_iu_i)\nabla u_i\cdot\nabla u_i-\lambda_iu_i^2\big]\\
  &\quad-\frac1p\int_\Omega\big[t_1^p|u_1|^p+2\beta t_1^{\frac p2}t_2^{\frac p2}|u_1|^{\frac p2}|u_2|^{\frac p2}+t_2^p|u_2|^p\big].
\end{align*}

The Nehari set $\Nb$ is non-empty. Indeed, if $u\in C_c^\infty(\Omega;\R^2)\cap \mathcal H$ and $u_1,u_2$ have disjoint supports, the function $h_u(t)$ has the same behavior as \[g(t)=c_1|t|^2-c_2|t|^p\] at infinity and also near the origin. Then there exists a maximizer $t_{\max}\in(0,+\infty)^2$. Thus, $t_{\max}u=(t_{\max,1}u_1,t_{\max,2}u_2)\in\Nb$.  

\begin{lemma}\label{coerciva}
    The functional ${\Jb}$ is positive and coercive on $\mathcal N_{\lambda,\beta}$ when $\lambda_1,\lambda_2<\frac{p-2-\gamma}{p-2}\nu\mu_1$.
\end{lemma}
\begin{proof}
    For each $u\in\mathcal N_{\lambda,\beta}$, we have
    \[
    \resizebox{\linewidth}{!}{$\displaystyle
    \begin{aligned}
        \Jb(u)&=\sum_{i=1}^2\bigg[\frac12\int_\Omega A_i(x,u_i)\nabla u_i\cdot\nabla u_i-\frac{\lambda_i}{2}\int_\Omega u_i^2-\frac1p\int_\Omega |u_i|^p\bigg]-\frac{2\beta}{p}\int_\Omega|u_1|^{\frac p2}|u_2|^{\frac p2}\\
        &=\left(\frac12-\frac1p\right)\sum_{i=1}^2\int_\Omega \left[A_i(x,u_i)\nabla u_i\cdot\nabla u_i-\lambda_i u_i^2\right]-\frac{1}{2p}\sum_{i=1}^2\int_\Omega\left( D_{u_i}A_i(x,u_i)\nabla u_i\cdot\nabla u_i\right) u_i\\
        &\ge\frac{p-2-\gamma}{2p}\sum_{i=1}^2\int_\Omega A_i(x,u_i)\nabla u_i\cdot\nabla u_i-\frac{p-2}{2p}\sum_{i=1}^2\int_\Omega \lambda_iu_i^2\\
        &\ge\frac{p-2-\gamma}{2p}\nu\sum_{i=1}^2\|u_i\|^2-\frac{p-2}{2p}\sum_{i=1}^2\int_\Omega \lambda_iu_i^2.
    \end{aligned}
    $}
    \]
    Since $\gamma<p-2$, the coercivity is clear for $\lambda_i\le0$. Otherwise, 
    \[\Jb(u)\ge\sum_{i=1}^2\left(\frac{p-2-\gamma}{2p}\nu-\frac{(p-2)\lambda_i}{2p\mu_1}\right)\|u_i\|^2,\]
    
   and the claim follows for \(\lambda_i<\frac{(p-2-\gamma)\nu\mu_1}{p-2}.\)
\end{proof}
\begin{lemma}\label{chiusura}
    Let $u\in\mathcal N_{\lambda,\beta}$ and assume that $\lambda_1,\lambda_2<\nu\mu_1$. Then there exist $C_1,C_2>0$ such that \[\int_\Omega|\nabla u_i|^2\ge C_1,\ \int_\Omega |u_i|^p\ge C_2,\ \ i=1,2.\]
    As a consequence, $\mathcal N_{\lambda,\beta}$ is closed w.r.t. the strong topology in $H_0^1(\Omega;\R^2)$ and this implies that $\mathcal N_{\lambda,\beta}$ is a complete metric space.
\end{lemma}
\begin{proof}
    Let $u=(u_1,u_2)\in\mathcal N_{\lambda,\beta}$. Then
    \[
    \resizebox{\linewidth}{!}{$\displaystyle
    \begin{aligned}
        \nu\|u_i\|^2&\le\int_\Omega A_i(x,u_i)\nabla u_i\cdot\nabla u_i\le\int_\Omega A_i(x,u_i)\nabla u_i\cdot\nabla u_i+\frac12\int_\Omega \left(D_{u_i}A_i(x,u_i)\nabla u_i\cdot\nabla u_i\right)u_i\\
        &\le\int_\Omega A_i(x,u_i)\nabla u_i\cdot\nabla u_i+\frac12\int_\Omega \left(D_{u_i}A_i(x,u_i)\nabla u_i\cdot\nabla u_i\right)u_i-\beta\int_\Omega|u_1|^{\frac p2}|u_2|^{\frac p2}\\
        &=\lambda_i\int_\Omega u_i^2+\int_\Omega |u_i|^p.
    \end{aligned}
    $}
    \]
  If $\lambda_i\le0$, by the Sobolev embedding, there exists $C>0$ such that
    \[\nu\|u_i\|^2\le C\|u_i\|^p,\]
    and there exist  $C_1,C_2>0$ such that
 \[\int_\Omega|\nabla u_i|^2\ge C_1,\qquad\int_\Omega |u_i|^p\ge C_2.\]
    If $0<\lambda_i<\nu\mu_1$ we have
    \[\nu\left(1-\frac{\lambda_i}{\nu\mu_1}\right)\|u_i\|^2\le\nu\int_\Omega|\nabla u_i|^2-\lambda_i\int_\Omega u_i^2\le\int_\Omega |u_i|^p,\]
   and again the claim follows by the Sobolev embedding.
\end{proof}

\begin{lemma}\label{unique critical point}
 For every $u=(u_1,u_2)\in \mathcal N_{\lambda,\beta}$, the function $h_u$ has at most one critical point.
\end{lemma}
\begin{proof}
Since $A_i(x,-s)=A_i(x,s)$, we have $h_u=h_{(|u_1|,|u_2|)}$ and $u\in\mathcal N_{\lambda,\beta}$ if and only if $(|u_1|,|u_2|)\in\mathcal N_{\lambda,\beta}$; hence we may assume that $u_1,u_2\ge0$. Since $u\in\mathcal N_{\lambda,\beta}$, $(1,1)$ is a critical point for $h_u$. Let $t\ne(1,1)$ be a second critical point.  Then 
  \begin{align*}
    0&=t_1\int_\Omega[ A_1(x,t_1u_1)\nabla u_1\cdot\nabla u_1-\lambda_1u_1^2]+\frac{t_1^2}{2}\int_\Omega \left(D_{s}A_1(x,t_1u_1)\nabla u_1\cdot\nabla u_1\right)u_1\\
    &\quad-t_1^{p-1}\int_\Omega |u_1|^p-\beta t_1^{\frac p2-1}t_2^{\frac p2}\int_\Omega|u_1|^{\frac p2}|u_2|^{\frac p2},\\
    0&=t_2\int_\Omega[ A_2(x,t_2u_2)\nabla u_2\cdot\nabla u_2-\lambda_2u_2^2]+\frac{t_2^2}{2}\int_\Omega \left(D_{s}A_2(x,t_2u_2)\nabla u_2\cdot\nabla u_2\right)u_2\\
    &\quad-t_2^{p-1}\int_\Omega |u_2|^p-\beta t_1^{\frac p2}t_2^{\frac p2-1}\int_\Omega|u_1|^{\frac p2}|u_2|^{\frac p2},
  \end{align*}
  where $D_sA_i(x,z):=D_zA_i(x,z)$ for every $z\in\R$ and $i=1,2$.
Using once again the fact that $u\in\Nb$:
\[
\resizebox{\linewidth}{!}{$\displaystyle
\begin{aligned}
    &\int_\Omega [t_1A_1(x,t_1u_1)-t_1^{p-1}A_1(x,u_1)]\nabla u_1\cdot\nabla u_1-(t_1-t_1^{p-1})\lambda_1\int_\Omega u_1^2\\
    +&\frac12\int_\Omega [t_1^2D_{s}A_1(x,t_1u_1)-t_1^{p-1}D_{s}A_1(x,u_1)]\nabla u_1\cdot\nabla u_1\,u_1=\beta(t_2^{\frac p2}t_1^{\frac p2-1}-t_1^{p-1})\int_\Omega |u_1|^{\frac p2}|u_2|^{\frac p2},
\end{aligned}
$}
\]
\[
\resizebox{\linewidth}{!}{$\displaystyle
\begin{aligned}
    &\int_\Omega [t_2A_2(x,t_2u_2)-t_2^{p-1}A_2(x,u_2)]\nabla u_2\cdot\nabla u_2-(t_2-t_2^{p-1})\lambda_2\int_\Omega u_2^2\\
    +&\frac12\int_\Omega [t_2^2D_{s}A_2(x,t_2u_2)-t_2^{p-1}D_{s}A_2(x,u_2)]\nabla u_2\cdot\nabla u_2\,u_2=\beta(t_1^{\frac p2}t_2^{\frac p2-1}-t_2^{p-1})\int_\Omega |u_1|^{\frac p2}|u_2|^{\frac p2}.
\end{aligned}
$}
\]
We write the left-hand side (in both the equations) as 
\begin{equation}\label{lhs}
\begin{aligned}
 t_i^{p-1}\bigg\{
 &\int_\Omega
 \left[\frac{A_i(x,t_iu_i)}{t_i^{p-2}}-A_i(x,u_i)\right]
 \nabla u_i\cdot\nabla u_i
 -\lambda_i\left(\frac{1}{t_i^{p-2}}-1\right)\int_\Omega u_i^2 \\
 &\quad+\frac12\int_\Omega
 \left[\frac{t_iu_iD_sA_i(x,t_iu_i)}{t_i^{p-2}}
 -u_iD_sA_i(x,u_i)\right]
 \nabla u_i\cdot\nabla u_i
 \bigg\},
\end{aligned}
\end{equation}
and we have 
\begin{align*}
\Phi_1(t_i)&=\int_\Omega
\frac{A_i(x,t_iu_i)\nabla u_i\cdot\nabla u_i-\lambda_i u_i^2}{t_i^{p-2}},\\
\Phi_2(t_i)&=\int_\Omega
\frac{t_iu_iD_sA_i(x,t_iu_i)}{t_i^{p-2}}
\nabla u_i\cdot\nabla u_i
\end{align*}
are strictly decreasing w.r.t. the variable $t_i$ provided that
$\lambda_i<\frac{p-2-\gamma}{p-2}\nu\mu_1$. Indeed, we compute
\[
\resizebox{\linewidth}{!}{$\displaystyle
\begin{aligned}
\Phi_1'(t_i)
&=t_i^{1-p}\int_\Omega\big[ t_iu_iD_sA_i(x,t_iu_i)\nabla u_i\cdot\nabla u_i
-(p-2)A_i(x,t_iu_i)\nabla u_i\cdot\nabla u_i
+(p-2)\lambda_i u_i^2\big]\\
&\le t_i^{1-p}\left[-(p-2-\gamma)\int_\Omega A_i(x,t_iu_i)\nabla u_i\cdot\nabla u_i
+(p-2)\lambda_i\int_\Omega u_i^2\right]\\
&\le t_i^{1-p}\left[-(p-2-\gamma)\nu\int_\Omega |\nabla u_i|^2
+(p-2)\lambda_i\int_\Omega u_i^2\right]\\
&\le t_i^{1-p}\left[-(p-2-\gamma)\nu\mu_1+(p-2)\lambda_i\right]
\int_\Omega u_i^2<0.
\end{aligned}
$}
\]
It remains to prove that $\Phi_2$ is decreasing in $t_i$. We write
\[
\Phi_2(t_i)=\int_\Omega t_i^{3-p}u_iD_sA_i(x,t_iu_i)\nabla u_i\cdot\nabla u_i.
\]
On $\{u_i>0\}$, this is written in terms of $(t_iu_i)^{3-p}D_sA_i(x,t_iu_i)$, which is decreasing by \eqref{a.3}, while the integrand vanishes on $\{u_i=0\}$.

Thus, the left-hand side \eqref{lhs} in the  $i$-th equation is positive if $t_i<1$ and it is negative when $t_i>1$. 
Let  $t_{\max}=\max\{t_1,t_2\}$ and $t_{\min}=\min\{t_1,t_2\}$, then the right-hand side is non-positive (and the left-hand side is positive) if $t_i=t_{\min}<1$ and it is non-negative if $t_i=t_{\max}>1$ (and the left-hand side is negative), hence we get a contradiction.
\end{proof}
\begin{remark}
Assumption \eqref{a.3} is needed only in the proof of Lemma \ref{unique critical point}.
\end{remark}
\noindent We point out that 
\begin{equation}\label{critical point and projection}
\begin{aligned}
&t=(t_1,t_2)\in(0,+\infty)^2\text{ is a critical point for }h_u\\
&\qquad\Longleftrightarrow\quad
tu=(t_1u_1,t_2u_2)\in\mathcal N_{\lambda,\beta}.
\end{aligned}
\end{equation}
We set 
\begin{align*}
   {\mathcal U}&:=\left\{u\in \mathcal H:  h_u\ \text{has a critical point}\right\} \\
   &=\left\{u\in \mathcal H:\ tu\in\mathcal N_{\lambda,\beta},\ \text{for some $t\in(0,+\infty)^2$} \right\}
\end{align*}
and $  \mathbb T:=\partial B_1\times \partial B_1$  the product of 2 spheres in $H_0^1(\Omega;\R)$.  We are following the procedure described in \cite{clapp2019simple} for the Laplacian, i.e. we transfer the minimizing problem on $\Nb$ to a problem on the $C^1$-manifold $\mathbb T$.
\begin{lemma}\label{max existence}
    Let $u=(u_1,u_2)\in\mathcal U$. The unique critical point of $h_u$ is the global maximizer.
\end{lemma}
\begin{proof}
    We assume that $u\in\mathcal N_{\lambda,\beta}$. Otherwise, we consider $\tilde u:=\tilde t u=(\tilde t_1u_1,\tilde t_2 u_2)$ and we note that $h_u(t)=h_{\tilde u}(t/\tilde t)$.  If $u\in\mathcal N_{\lambda,\beta}$, the map $h_u$ has the unique critical point $t=(1,1)$. We claim that $h_u$ attains the global maximum in a compact set $K\subset(0,+\infty)^2$. Indeed, 
    \[h_u(s)\ge\sum_{i=1}^2\left[\frac{s_i^2}{2}\int_\Omega[ A_i(x,s_iu_i)\nabla u_i\cdot\nabla u_i-\lambda_iu_i^2]-\frac{s_i^p}{p}\int_\Omega|u_i|^p\right].\]
   Hence, $(0,0)$ is a local minimizer.
Moreover, 
\[
\begin{aligned}
h_u(s_1,s_2)-h_u(s_1,0)
&\ge\frac{s_2^2}{2}\left(\int_\Omega A_2(x,s_2u_2)
\nabla u_2\cdot\nabla u_2-\lambda_2\int_\Omega u_2^2\right)\\
&\quad-\frac{s_2^p}{p}\int_\Omega|u_2|^p>0,
\end{aligned}
\]
for \(s_2>0\) sufficiently small. By symmetry, the analogous estimate holds near the other axis. Hence no maximizer can lie on \(\partial([0,+\infty)^2)\).

   Additionally,  since $u\in\Nb$, we have 
   \begin{align}\label{condizione max}
       \int_\Omega|u_i|^p+\beta\int_\Omega |u_1|^{\frac p2}|u_2|^{\frac p2}>0,\ \ i=1,2.
   \end{align}
Now, we consider the function
\[
\begin{aligned}
h_u(s)&=\sum_{i=1}^2\left[\frac{s_i^2}{2}\int_\Omega
[A_i(x,s_iu_i)\nabla u_i\cdot\nabla u_i-\lambda_iu_i^2]
-\frac{s_i^p}{p}\int_\Omega |u_i|^p\right]\\
&\quad-\frac{2\beta}{p}s_1^{\frac p2}s_2^{\frac p2}
\int_\Omega|u_1|^{\frac p2}|u_2|^{\frac p2}.
\end{aligned}
\]
We treat the higher-order term. Since $u\in\Nb$, the two Nehari identities imply
\[
\int_\Omega |u_i|^p+
\beta\int_\Omega |u_1|^{\frac p2}|u_2|^{\frac p2}>0,
\qquad i=1,2.
\]
If
\[
\int_\Omega |u_1|^{\frac p2}|u_2|^{\frac p2}=0,
\]
then the $p$-order part is simply
\[
s_1^p\int_\Omega|u_1|^p+s_2^p\int_\Omega|u_2|^p
\]
and it is positive definite. Otherwise, since $\beta<0$, the previous inequalities yield
\[
\int_\Omega|u_1|^p>-\beta\int_\Omega |u_1|^{\frac p2}|u_2|^{\frac p2},
\qquad
\int_\Omega|u_2|^p>-\beta\int_\Omega |u_1|^{\frac p2}|u_2|^{\frac p2},
\]
and therefore
\[
\left(\int_\Omega|u_1|^p\right)
\left(\int_\Omega|u_2|^p\right)>
\beta^2\left(\int_\Omega |u_1|^{\frac p2}|u_2|^{\frac p2}\right)^2.
\]
Thus the quadratic form associated with the variables $s_1^{\frac p2}$ and $s_2^{\frac p2}$ is positive definite. Hence there exists $\delta>0$ such that
\[
s_1^p\int_\Omega|u_1|^p
+2\beta s_1^{\frac p2}s_2^{\frac p2}\int_\Omega |u_1|^{\frac p2}|u_2|^{\frac p2}
+s_2^p\int_\Omega|u_2|^p
\ge\delta(s_1^p+s_2^p).
\]
On the other hand, by hypothesis \eqref{a.0}, there exists $C>0$ such that the quadratic part of $h_u$ is bounded from above by $C(s_1^2+s_2^2)$. Since $p>2$, the positive $p$-order term dominates this quadratic part at infinity, and consequently $h_u(s)\to-\infty$ as $|s|\to+\infty$. This geometry ensures the existence of the maximizer in $(0,+\infty)^2$, which is unique by Lemma \ref{unique critical point}.  
\end{proof}
We define $ t_u=(t_{u,1},t_{u,2}):=\displaystyle{\arg\max_{t\in(0,+\infty)^2}}h_u(t)$ the maximizer found with Lemmas \ref{unique critical point}, \ref{max existence}, and  the map
\begin{align*}
\tilde{\mathfrak m}:\ \mathcal U&\to\mathcal N_{\lambda,\beta}\\
u=(u_1,u_2)&\mapsto t_uu=(t_{u,1}u_1,t_{u,2}u_2).
\end{align*}

\begin{lemma}
    The set $\mathcal U\subset H_0^1(\Omega;\R^2)$ is open in $H_0^1(\Omega;\R^2)$ and the map $\mathcal U\ni u\mapsto t_u\in(0,+\infty)^2$ is continuous with respect to the strong topology of $H_0^1(\Omega;\R^2)$.
\end{lemma}
\begin{proof}
    Let $u=(u_1,u_2)\in\mathcal U$ and we consider a ball $B_\varepsilon(u)$. As in Lemma \ref{max existence}, we can assume $u\in\mathcal N_{\lambda,\beta}$ and $t_u=(1,1)$ is the global maximizer of $h_u.$ Let $v=v^\varepsilon\in B_\varepsilon(u)$. We recall that:
    \begin{align*}
       \int_\Omega|u_i|^p>-\beta\int_\Omega|u_1|^{\frac p2}|u_2|^{\frac p2},\ \ i=1,2,
   \end{align*}
   by the definition of $\Nb$.
    Since $\|u-v^\varepsilon\|<\varepsilon$, the same inequality holds for $v^\varepsilon$. Indeed, if $v^\varepsilon$ satisfies
    \[\int_\Omega |v^\varepsilon|^p+\beta\int_\Omega |v_1^\varepsilon|^{\frac p2}|v_2^\varepsilon|^{\frac p2}\le0,\]
    passing to the limit as $\varepsilon\to0^+$ we get a contradiction and the claim follows.
    
    As in Lemma \ref{max existence}, there exists the global maximizer for $h_v$ and it is the unique critical point $t_v$ (by Lemma \ref{unique critical point}). Thus, $t_v v\in\mathcal N_{\lambda,\beta}$ and $v\in\mathcal U$.

We now prove that the map $u\mapsto t_u$ is continuous with respect to the strong topology of $H_0^1(\Omega;\mathbb R^2)$. Let $v=u_n$ with $u_n\to u$ in $H_0^1(\Omega;\R^2)$ and let $\{t_{u_n}\}$ be the sequence of the maximizers of $h_{u_n}$. We claim that $\{t_{u_n}\}$ is bounded in $\R^2$. Otherwise, we can assume $|t_{u_n}|\to+\infty$ up to a subsequence.  Since $t_{u_n}u_n\in\Nb$ we have: 
    \begin{align}
    0&=t_{u_{n},1}\int_\Omega[ A_1(x,t_{u_{n},1}u_{n,1})\nabla u_{n,1}\cdot\nabla u_{n,1}-\lambda_1u_{n,1}^2]\notag\\
    &\quad+\frac{t_{u_{n},1}^2}{2}\int_\Omega u_{n,1}D_{s}A_1(x,t_{u_{n},1}u_{n,1})\nabla u_{n,1}\cdot\nabla u_{n,1}\notag\\
    &\quad-t_{u_{n},1}^{p-1}\int_\Omega |u_{n,1}|^p-\beta t_{u_n,2}^{\frac p2}t_{u_{n},1}^{\frac p2-1}\int_\Omega |u_{n,2}|^{\frac p2}|u_{n,1}|^{\frac p2},\label{nehari-tun-1}
  \end{align}
  and
   \begin{align}
    0&=t_{u_{n},2}\int_\Omega[ A_2(x,t_{u_{n},2}u_{n,2})\nabla u_{n,2}\cdot\nabla u_{n,2}-\lambda_2u_{n,2}^2]\notag\\
    &\quad+\frac{t_{u_{n},2}^2}{2}\int_\Omega u_{n,2}D_{s}A_2(x,t_{u_{n},2}u_{n,2})\nabla u_{n,2}\cdot\nabla u_{n,2}\notag\\
    &\quad-t_{u_{n},2}^{p-1}\int_\Omega |u_{n,2}|^p-\beta t_{u_n,1}^{\frac p2}t_{u_{n},2}^{\frac p2-1}\int_\Omega |u_{n,1}|^{\frac p2}|u_{n,2}|^{\frac p2},\label{nehari-tun-2}
  \end{align}
  where $D_sA_i(x,z):=D_zA_i(x,z)$ for every $z\in\R$ and $i=1,2$.
Set $\rho_n:=|t_{u_n}|$ and $s_{n,i}:=t_{u_n,i}/\rho_n$, $i=1,2$. Up to a subsequence, $s_n\to s=(s_1,s_2)$, with $s_i\ge0$ and $|s|=1$. Multiplying \eqref{nehari-tun-1}-\eqref{nehari-tun-2} by $t_{u_n,i}$ and using \eqref{a.0} and \eqref{a.2}, we get, for $i=1,2$,
\[
t_{u_n,i}^p\int_\Omega |u_{n,i}|^p
+\beta t_{u_n,1}^{\frac p2}t_{u_n,2}^{\frac p2}
\int_\Omega |u_{n,1}|^{\frac p2}|u_{n,2}|^{\frac p2}
\le C t_{u_n,i}^2.
\]
Dividing by $\rho_n^p$ and passing to the limit, we obtain
\[
s_i^p\int_\Omega |u_i|^p
+\beta s_1^{\frac p2}s_2^{\frac p2}
\int_\Omega |u_1|^{\frac p2}|u_2|^{\frac p2}\le0,
\qquad i=1,2,
\]
Since $u\in\mathcal U$, we can assume that $u\in\Nb$, we have
\[
\int_\Omega |u_i|^p
+\beta\int_\Omega |u_1|^{\frac p2}|u_2|^{\frac p2}>0,
\qquad i=1,2,
\]
and hence
\begin{equation}\label{strict-ineq-limiting-system}
\left(\int_\Omega |u_1|^p\right)
\left(\int_\Omega |u_2|^p\right)
>\beta^2\left(\int_\Omega |u_1|^{\frac p2}|u_2|^{\frac p2}\right)^2.
\end{equation}
 This contradicts the limiting system: if one of $s_1,s_2$ is zero the contradiction is immediate. If both are positive, then, dividing the two limiting inequalities by $s_1^{\frac p2}$ and $s_2^{\frac p2}$, respectively, we obtain
\[
 s_1^{\frac p2}\int_\Omega |u_1|^p
 +\beta s_2^{\frac p2}\int_\Omega |u_1|^{\frac p2}|u_2|^{\frac p2}\le0,
\]
and
\[
 s_2^{\frac p2}\int_\Omega |u_2|^p
 +\beta s_1^{\frac p2}\int_\Omega |u_1|^{\frac p2}|u_2|^{\frac p2}\le0.
\]
Since $\beta<0$, these inequalities imply
\[
 s_1^{\frac p2}\int_\Omega |u_1|^p
 \le -\beta s_2^{\frac p2}\int_\Omega |u_1|^{\frac p2}|u_2|^{\frac p2},
\qquad
 s_2^{\frac p2}\int_\Omega |u_2|^p
 \le -\beta s_1^{\frac p2}\int_\Omega |u_1|^{\frac p2}|u_2|^{\frac p2}.
\]
Multiplying them and simplifying the positive factor $s_1^{\frac p2}s_2^{\frac p2}$, we get
\[
\left(\int_\Omega |u_1|^p\right)
\left(\int_\Omega |u_2|^p\right)
\le \beta^2\left(\int_\Omega |u_1|^{\frac p2}|u_2|^{\frac p2}\right)^2,
\]
which contradicts \eqref{strict-ineq-limiting-system}. Thus, $\{t_{u_n}\}$ is bounded and, up to a subsequence, \(t_{u_n}\to\bar t\in[0,+\infty)^2\). We claim that $\bar t_i>0$ for every $i=1,2$. Indeed, assume by contradiction that $\bar t_i=0$ for some $i$. Dividing the corresponding identity in \eqref{nehari-tun-1}-\eqref{nehari-tun-2} by $t_{u_n,i}$, we have
\begin{align*}
0&=\int_\Omega A_i(x,t_{u_n,i}u_{n,i})\nabla u_{n,i}\cdot\nabla u_{n,i}-\lambda_i\int_\Omega u_{n,i}^2\\
&\quad+\frac{t_{u_n,i}}{2}\int_\Omega u_{n,i}D_sA_i(x,t_{u_n,i}u_{n,i})\nabla u_{n,i}\cdot\nabla u_{n,i}-t_{u_n,i}^{p-2}\int_\Omega |u_{n,i}|^p\\
&\quad-\beta t_{u_n,j}^{\frac p2}t_{u_n,i}^{\frac p2-2}\int_\Omega |u_{n,j}|^{\frac p2}|u_{n,i}|^{\frac p2},
\end{align*}
where $j\ne i$. Using \eqref{a.2} and $\beta<0$, we get
\[
\int_\Omega A_i(x,t_{u_n,i}u_{n,i})\nabla u_{n,i}\cdot\nabla u_{n,i}
-\lambda_i\int_\Omega u_{n,i}^2
\le t_{u_n,i}^{p-2}\int_\Omega |u_{n,i}|^p.
\]
Passing to the limit, since $u_i\not\equiv0$ and $\lambda_i<\nu\mu_1$, we obtain
\[
0<\left(\nu-\frac{\lambda_i}{\mu_1}\right)\|u_i\|^2\le0,
\]
a contradiction. Hence \(\bar t\in(0,+\infty)^2\). Therefore \(\bar{t}u\in \Nb\). The uniqueness of the critical point, from Lemma \ref{unique critical point}, and \eqref{critical point and projection} imply also that \(\bar{t} = t_u\).
    
\end{proof}
\begin{corollary}
    The map $\mathfrak m:=\tilde{\mathfrak m}_{|_{\mathcal U\cap\mathbb T}}:\mathcal U\cap\mathbb T\to\mathcal N_{\lambda,\beta}$ is a homeomorphism with inverse
    \[\mathfrak m^{-1}(u)=\left(\frac{u_1}{\|u_1\|},\frac{u_2}{\|u_2\|}\right),\]
for every $u=(u_1,u_2)\in\Nb$.
\end{corollary}
\begin{remark}\label{remark boundary U}
Let $\lambda_1,\lambda_2<\nu\mu_1$. We claim that $\mathcal U\cap\mathbb T\ne\mathbb T$ provided that $\beta<-1$. Let $(v,v)\in\mathbb T$ and we assume that there exists $t\in(0,+\infty)^2$ such that  $(t_1v,t_2v)\in\mathcal N_{\lambda,\beta}$. Then
    \begin{align}\label{first inequality}
        0<t_1^2\int_\Omega[ A_1(x,t_1v)\nabla v\cdot\nabla v-\lambda_1v^2]\le (t_1^p+\beta t_1^{\frac p2}t_2^{\frac p2})\int_\Omega |v|^p
    \end{align}
   and also
\begin{align}\label{second inequality}
   0<t_2^2\int_\Omega[A_2(x,t_2v)\nabla v\cdot\nabla v-\lambda_2v^2]\le (t_2^p+\beta t_1^{\frac p2}t_2^{\frac p2})\int_\Omega |v|^p.
\end{align}
   If $t_1\le t_2$ and $\beta<-1$, we get a contradiction in \eqref{first inequality}, otherwise we conclude using \eqref{second inequality}. 
     In particular, $\mathcal U\cap\mathbb T$ has boundary $\partial\mathcal U\cap\mathbb T$.

    Moreover, since $\mathcal N_{\lambda,\beta}$ is closed, if $\{u_n\}\subset \mathcal U\cap\mathbb T$ and $u_n\to u\in\partial \mathcal U\cap\mathbb T$, then $\|\mathfrak m(u_n)\|\to+\infty$. Indeed, write $\mathfrak m(u_n)=t_nu_n$, with $t_n=(t_{n,1},t_{n,2})\in(0,+\infty)^2$. If $\|\mathfrak m(u_n)\|$ were bounded, then $\{t_n\}$ would be bounded and, up to a subsequence, $t_n\to\bar t$. Hence $\mathfrak m(u_n)=t_nu_n\to\bar t u$, and the closedness of $\Nb$ would give $\bar t u\in\Nb$. Therefore $u\in\mathcal U$, contradicting $u\in\partial\mathcal U\cap\mathbb T$.
\end{remark}
Finally, we define the extended functional
\[
\widetilde{\mathcal I}_{\lambda,\beta}:\mathcal U\to\R,
\qquad
\widetilde{\mathcal I}_{\lambda,\beta}(u):=\Jb(\tilde{\mathfrak m}(u)).
\]
Then $\mathcal I_{\lambda,\beta}:\mathcal U\cap\mathbb T\to\R$ is the restriction of $\widetilde{\mathcal I}_{\lambda,\beta}$ to $\mathcal U\cap\mathbb T$, namely $\mathcal I_{\lambda,\beta}:=\Jb\circ\mathfrak m$, and
\[
\begin{aligned}
\mathcal I_{\lambda,\beta}(v)
&=\frac{p-2}{2p}\sum_{i=1}^2\int_\Omega
\left[A_i(x,v_i)\nabla v_i\cdot\nabla v_i-\lambda_i v_i^2\right]\\
&\quad-\frac{1}{2p}\sum_{i=1}^2\int_\Omega
\left(D_{v_i}A_i(x,v_i)\nabla v_i\cdot\nabla v_i\right)v_i.
\end{aligned}
\]
Moreover,
\[l_{\lambda,\beta}:=\inf_{u\in\mathcal N_{\lambda,\beta}}\Jb(u)=\inf_{v\in\mathcal U\cap\mathbb T}\mathcal I_{\lambda,\beta}(v).\]
We denote by $T_v(\mathbb T)$ the tangent space of $\mathbb T$ at some $v\in\mathbb T$: 
\[T_v(\mathbb T):=\left\{\varphi=(\varphi_1,\varphi_2)\in H_0^1(\Omega;\R^2):\ \langle \varphi_1, v_1\rangle_{H_0^1}=0,\ \langle \varphi_2, v_2\rangle_{H_0^1}=0\right\},\]
and we study the differentiability of the extended functional $\widetilde{\mathcal I}_{\lambda,\beta}$.
\begin{lemma}\label{I differentiable}
   Let $u\in\mathcal U$. The functional $\widetilde{\mathcal I}_{\lambda,\beta}$ is G\^ateaux differentiable along any direction $\varphi\in C_c^\infty(\Omega;\R^2)$. In particular, we have
   \[
   \langle\widetilde{\mathcal I}_{\lambda,\beta}'(u),\varphi\rangle=\langle\Jb'(\tilde{\mathfrak m}(u)),t_u\varphi\rangle.
  \]
   Consequently, if $v\in\mathcal U\cap\mathbb T$, then
   \[
   \langle\mathcal I_{\lambda,\beta}'(v),\varphi\rangle=\langle\Jb'(\mathfrak m(v)),t_v\varphi\rangle,
  \]
    for every $\varphi\in T_v(\mathbb T)\cap C_c^\infty(\Omega;\R^2)$.
\end{lemma}
\begin{proof}
    Let $u\in\mathcal U$ and $\varphi\in C_c^\infty(\Omega;\R^2)$. Since $\mathcal U$ is open, we have $u+\delta\varphi\in\mathcal U$ for $|\delta|$ small enough. Moreover, $\tilde{\mathfrak m}(u)=t_uu$, where $t_u$ is the maximizer of $h_u$. By the Mean Value Theorem, we have
    \begin{align*}
        \widetilde{\mathcal I}_{\lambda,\beta}(u+\delta\varphi)-\widetilde{\mathcal I}_{\lambda,\beta}(u)&=\Jb(t_{u+\delta\varphi}(u+\delta\varphi))-\Jb(t_uu)\\
        &\le\Jb(t_{u+\delta\varphi}(u+\delta\varphi))-\Jb(t_{u+\delta\varphi}u)\\
        &=\langle\Jb'(t_{u+\delta\varphi}(u+\theta_1\delta\varphi)),\delta t_{u+\delta\varphi}\varphi\rangle,
    \end{align*}
    for $|\delta|$ small enough and $\theta_1\in(0,1)$. Similarly, 
    \[
    \begin{aligned}
    \widetilde{\mathcal I}_{\lambda,\beta}(u+\delta\varphi)
    -\widetilde{\mathcal I}_{\lambda,\beta}(u)
    &\ge\Jb(t_{u}(u+\delta\varphi))-\Jb(t_uu)\\
    &=\langle\Jb'(t_u(u+\theta_2\delta\varphi)),
    \delta t_u\varphi\rangle,
    \end{aligned}
    \]
    for some $\theta_2\in(0,1)$. Passing to the limit as $\delta\to0$, and using the continuity of $u\mapsto t_u$, we obtain
    \begin{align*}
        \lim_{\delta\to0}\frac{ \widetilde{\mathcal I}_{\lambda,\beta}(u+\delta\varphi)-\widetilde{\mathcal I}_{\lambda,\beta}(u)}{\delta}=\langle\Jb'(t_uu),t_u\varphi\rangle.
    \end{align*}
    If $v\in\mathcal U\cap\mathbb T$ and $\varphi\in T_v(\mathbb T)\cap C_c^\infty(\Omega;\R^2)$, then the above identity restricts to the tangent directions of $\mathbb T$ because $\mathcal I_{\lambda,\beta}=\widetilde{\mathcal I}_{\lambda,\beta}|_{\mathcal U\cap\mathbb T}$.
\end{proof}
\begin{theorem}\label{PS sequence}
    There exists a minimizing sequence $\{v_n\}\subset\mathcal U\cap\mathbb T$ for $\mathcal I_{\lambda,\beta}$ such that, setting $u_n:=\mathfrak m(v_n)=t_{v_n}v_n$, we have:
    \begin{enumerate}
        \item[$(i)$] $\{v_n\}$ is a $(PS)$-sequence for $\mathcal I_{\lambda,\beta}$ at level $l_{\lambda,\beta}$.
        \item[$(ii)$] $\{u_n\}$ is a $(CPS)$-sequence for $\Jb$ at level $l_{\lambda,\beta}$.
    \end{enumerate}
\end{theorem}
\begin{proof}
Choose $M>l_{\lambda,\beta}$. By Remark \ref{remark boundary U} and the coercivity of $\Jb$ on $\Nb$, the sublevel set
\[
X_M:=\left\{v\in\mathcal U\cap\mathbb T:\ \mathcal I_{\lambda,\beta}(v)\le M\right\}
\]
is bounded away from $\partial\mathcal U\cap\mathbb T$. Hence $X_M$ is a closed subset of the complete metric space $\mathbb T$, and therefore it is complete. Applying Ekeland's variational principle to $\mathcal I_{\lambda,\beta}$ restricted to $X_M$, we can choose a minimizing sequence $\{v_n\}\subset X_M\subset\mathcal U\cap\mathbb T$ such that
\begin{itemize}
  \item $\mathcal I_{\lambda,\beta}(v_n)\le l_{\lambda,\beta}+\frac1n$,
  \item $\mathcal I_{\lambda,\beta}(v_n)-\mathcal I_{\lambda,\beta}(z)\le \frac1n\|v_n-z\|$ for every $z\in X_M$.
\end{itemize}
Since $\mathcal I_{\lambda,\beta}(v_n)<M$ for $n$ large, the above variational inequality holds locally in $\mathcal U\cap\mathbb T$. Thus
\[
|d\mathcal I_{\lambda,\beta}|(v_n)\to0,
\]
and $\{v_n\}$ is a $(PS)$-sequence for $\mathcal I_{\lambda,\beta}$ at level $l_{\lambda,\beta}$.

Let $t_n=(t_{n,1},t_{n,2})=t_{v_n}\in(0,+\infty)^2$ such that $\mathfrak m(v_n)=t_nv_n\in\mathcal N_{\lambda,\beta}$. 
{Set $u_n=\mathfrak m(v_n)=t_nv_n$. Since
$\Jb(u_n)=\mathcal I_{\lambda,\beta}(v_n)$ is bounded and
$u_n\in\mathcal N_{\lambda,\beta}$, the coercivity of $\Jb$ on
$\mathcal N_{\lambda,\beta}$ implies that $\{u_n\}$ is bounded in
$H_0^1(\Omega;\R^2)$. As $v_n\in\mathbb T$, it follows that
$\{t_n\}$ is bounded in $\R^2$; moreover, by Lemma \ref{chiusura},
$\|u_{n,i}\|\ge\delta>0$, hence $t_{n,i}\ge\delta>0$ for $i=1,2$.}
It remains to prove $(ii)$. By Lemma \ref{I differentiable}, we obtain:
 \[\langle\mathcal I_{\lambda,\beta}'(v_n),\varphi\rangle=\langle\Jb'(t_{v_n}v_n),t_{v_n}\varphi\rangle.\]
From Theorem \ref{theorem 1} and reasoning as in Theorem \ref{differentiability Jb}, we deduce that 
\[
\resizebox{\linewidth}{!}{$\displaystyle
|d\mathcal I_{\lambda,\beta}|(v_n)
\ge\sup_{\substack{\varphi\in T_{v_n}(\mathbb T)\cap C_c^\infty(\Omega;\R^2)\\ \|\varphi\|_{H_0^1(\Omega;\R^2)}\le1}}
\langle\mathcal I_{\lambda,\beta}'(v_n),\varphi\rangle
=\sup_{\substack{\varphi\in T_{v_n}(\mathbb T)\cap C_c^\infty(\Omega;\R^2)\\ \|\varphi\|_{H_0^1(\Omega;\R^2)}\le1}}
\langle\Jb'(t_{v_n}v_n),t_{v_n}\varphi\rangle.
$}
\]
Moreover, since $|d\mathcal I_{\lambda,\beta}|(v_n)\to0$, we have
 \begin{equation}\label{eq:tangent-control}
o(1)= |d\mathcal I_{\lambda,\beta}|(v_n)\ge\sup_{\substack{\varphi\in T_{v_n}(\mathbb T)\cap C_c^\infty(\Omega;\R^2)\\ \|\varphi\|_{H_0^1(\Omega;\R^2)}\le1}}\langle\Jb'(t_{v_n}v_n),\varphi\rangle.
\end{equation}

 Let $\psi\in H_0^1(\Omega;\R^2)$.
 Recalling that $H_0^1(\Omega;\R^2)=T_{v_n}(\mathbb T)\oplus\text{span}\{(v_{n,1},0),(0,v_{n,2})\}$, we can write 
 \[
 \psi=\xi+(\tau_1 v_{n,1}, \tau_2 v_{n,2})
\]
with $\xi\in T_{v_n}(\mathbb T)$ and $\tau=(\tau_1,\tau_2)\in\R^2$.
 By a density argument, \eqref{eq:tangent-control} also holds for $\xi$.
 Therefore,
 \begin{align*}
     o(1)&=\int_\Omega A_1(x,t_{v_n,1}v_{n,1})\nabla(t_{v_n,1}v_{n,1})\cdot\nabla\xi_1\\
     &\quad+\frac12\int_\Omega \left(D_sA_1(x,t_{v_n,1}v_{n,1})\nabla(t_{v_n,1}v_{n,1})\cdot\nabla(t_{v_n,1}v_{n,1})\right)\xi_1\\
     &\quad-\lambda_1\int_\Omega t_{v_n,1}v_{n,1}\xi_1-\int_\Omega g_{\beta,1}(t_{v_n}v_n)\xi_1,\\
     o(1)&=\int_\Omega A_2(x,t_{v_n,2}v_{n,2})\nabla(t_{v_n,2}v_{n,2})\cdot\nabla\xi_2\\
     &\quad+\frac12\int_\Omega \left(D_sA_2(x,t_{v_n,2}v_{n,2})\nabla(t_{v_n,2}v_{n,2})\cdot\nabla(t_{v_n,2}v_{n,2})\right)\xi_2\\
     &\quad-\lambda_2\int_\Omega t_{v_n,2}v_{n,2}\xi_2-\int_\Omega g_{\beta,2}(t_{v_n}v_n)\xi_2,
     \end{align*}
     where $D_sA_i(x,z):=D_zA_i(x,z)$ for every $z\in\R$ and $i=1,2$.
     Since $t_{v_n}v_n\in\Nb$, we have 
        \begin{align*}
         0&=\int_\Omega A_1(x,t_{v_n,1}v_{n,1})\nabla(t_{v_n,1}v_{n,1})\cdot\nabla v_{n,1}\\
         &\quad+\frac12\int_\Omega \left(D_sA_1(x,t_{v_n,1}v_{n,1})\nabla(t_{v_n,1}v_{n,1})\cdot\nabla(t_{v_n,1}v_{n,1})\right)v_{n,1}\\
     &\quad-\lambda_1\int_\Omega t_{v_n,1}v_{n,1}^2-\int_\Omega g_{\beta,1}(t_{v_n}v_n)v_{n,1},  \\
      0&=\int_\Omega A_2(x,t_{v_n,2}v_{n,2})\nabla(t_{v_n,2}v_{n,2})\cdot\nabla v_{n,2}\\
         &\quad+\frac12\int_\Omega \left(D_sA_2(x,t_{v_n,2}v_{n,2})\nabla(t_{v_n,2}v_{n,2})\cdot\nabla(t_{v_n,2}v_{n,2})\right)v_{n,2}\\
     &\quad-\lambda_2\int_\Omega t_{v_n,2}v_{n,2}^2-\int_\Omega g_{\beta,2}(t_{v_n}v_n)v_{n,2}.   
        \end{align*}
 Thus, $\{t_{v_n}v_n\}$ is a $(CPS)$-sequence for $\Jb$.
\end{proof}
\begin{proof}[Proof of Theorem \ref{main thm competitive quasilinear}]
    Let $\{v_n\}\subset\mathcal U\cap\mathbb T$ be the minimizing sequence given by Theorem \ref{PS sequence}. Since $\mathcal I_{\lambda,\beta}$ and $\mathcal U\cap\mathbb T$ are symmetric, we may suppose that $v_n\ge0$. Then, $u_n=\mathfrak m(v_n)\ge0$ is a $(CPS)$-sequence for $\Jb$ at level $l_{\lambda,\beta}$. By Theorem \ref{PS and boundeness} and Lemma \ref{PS1}, the functional $\Jb$ satisfies the $(CPS)$-condition also for $\beta<0,\lambda_1,\lambda_2<\frac{p-2-\gamma}{p-2}\nu\mu_1$ and we can assume, up to a subsequence, that $u_n\to u$. Lemma \ref{chiusura} implies that $u\in\Nb$. Since the strong limit of a $(CPS)$-sequence satisfies the corresponding Euler equation, $u$ is a weak solution of \eqref{Qb}.
\end{proof}


\end{document}